\documentclass{article}
\usepackage[utf8]{inputenc}
\usepackage{amsmath}
\usepackage{amssymb}
\usepackage{graphicx}
\usepackage{dsfont}
\usepackage{upgreek}
\usepackage{textcomp}
\usepackage{braket}
\usepackage[margin=1in]{geometry}
\usepackage{amsthm}
\usepackage{mathrsfs}
\usepackage{mathtools}
\usepackage[table,svgnames]{xcolor}
\usepackage{graphicx}
\usepackage{tikz}
\usepackage{float}
\usetikzlibrary{arrows}
\usepackage[toc,page]{appendix}
\usepackage{enumitem}

\usepackage{cleveref}

\crefname{appsec}{appendix}{appendices}

\usepackage[mathlines]{lineno}

\newcommand{\Real}{\mathrm{Re}}

\numberwithin{equation}{section}

\newtheoremstyle{newdefinition}
{20pt}
{20pt}
{}
{}
{\bfseries}
{.}
{.5em}
{}

\newtheorem{theorem}{Theorem}[section]
\newtheorem{lemma}{Lemma}[section]

\newtheorem{proposition}{Proposition}[section]

\newtheorem{hypothesis}{Hypothesis}
\Crefname{hypothesis}{Hypothesis}{Hypotheses}
\crefname{hypothesis}{hypothesis}{hypotheses}

\theoremstyle{newdefinition}
\newtheorem{definition}{Definition}[section]

\theoremstyle{remark}
\newtheorem{remark}{Remark}[section]

\title{Existence of Strong Solution for the Complexified Non-linear Poisson Boltzmann Equation}
\author{Brian Choi, Jie Xu, Trevor Norton, Mark Kon, Julio E. Castrill\'{o}n-Cand\'{a}s}
\date{}
\begin{document}
\maketitle
\begin{abstract}
We prove the existence and uniqueness of the complexified Nonlinear
Poisson-Boltzmann Equation (nPBE) in a bounded domain in
$\mathbb{R}^3$.  The nPBE is a model equation in nonlinear
electrostatics.  The standard convex optimization argument to the
complexified nPBE no longer applies, but instead, a contraction
mapping argument is developed. Furthermore, we show that uniqueness
can be lost if the hypotheses given are not satisfied.  The
complixified nPBE is highly relevant to regularity analysis of the
solution of the real nPBE with respect to the dielectric (diffusion)
and Debye-H\"uckel coefficients.  This approach is also well-suited to
investigate the existence and uniqueness problem for a wide class of
semi-linear elliptic Partial Differential Equations (PDEs).
\end{abstract}

\medskip


\section{Introduction.}

Linear elliptic partial differential equations have long been used to
model problems in physics, engineering, biology, and chemistry
\cite{evans2010partial}. In particular, simple linear elliptic
equations are often used to model the potential field generated by
molecular structures embedded in a solvent in thermal
equilibrium. However, a more accurate representation of this is given
by the non-linear Poisson Boltzmann Equation (nPBE). The nPBE is given
by
\begin{equation}\label{npb}
\begin{aligned}
	-\nabla \cdot (\epsilon(x) \nabla u) + \kappa(x)^2 \sinh u &=
        f, & &x\in \Omega,\\ u&= g, & &x \in \partial \Omega,
\end{aligned}
\end{equation}
where \(u\) is the nondimensionalized potential, \(\epsilon\) is the
dielectric, and \( \kappa ^2\) is the Debye-H\"uckel parameter
\cite{Holst1994}. The nPBE has found important applications in protein
interactions and molecular dynamics
\cite{Padhorny2016,Neumaier1997}. In \cref{nPBE:fig1} an example of
the electrostatic potential field is rendered from the solution of the
nPBE by using the Adaptive Poisson Boltzmann Solver (APBS)
\cite{Baker2001} for E. Coli RHo Protein. (PDB: 1A63
\cite{Berman2000}).  However, the mathematical properties of the nPBE
are less understood and significantly more complicated than the linear
case.

\begin{figure}[htb]
	\centering \includegraphics[height = 8cm, width = 8cm,
          trim=6cm 6cm 6cm 6cm,
          clip]{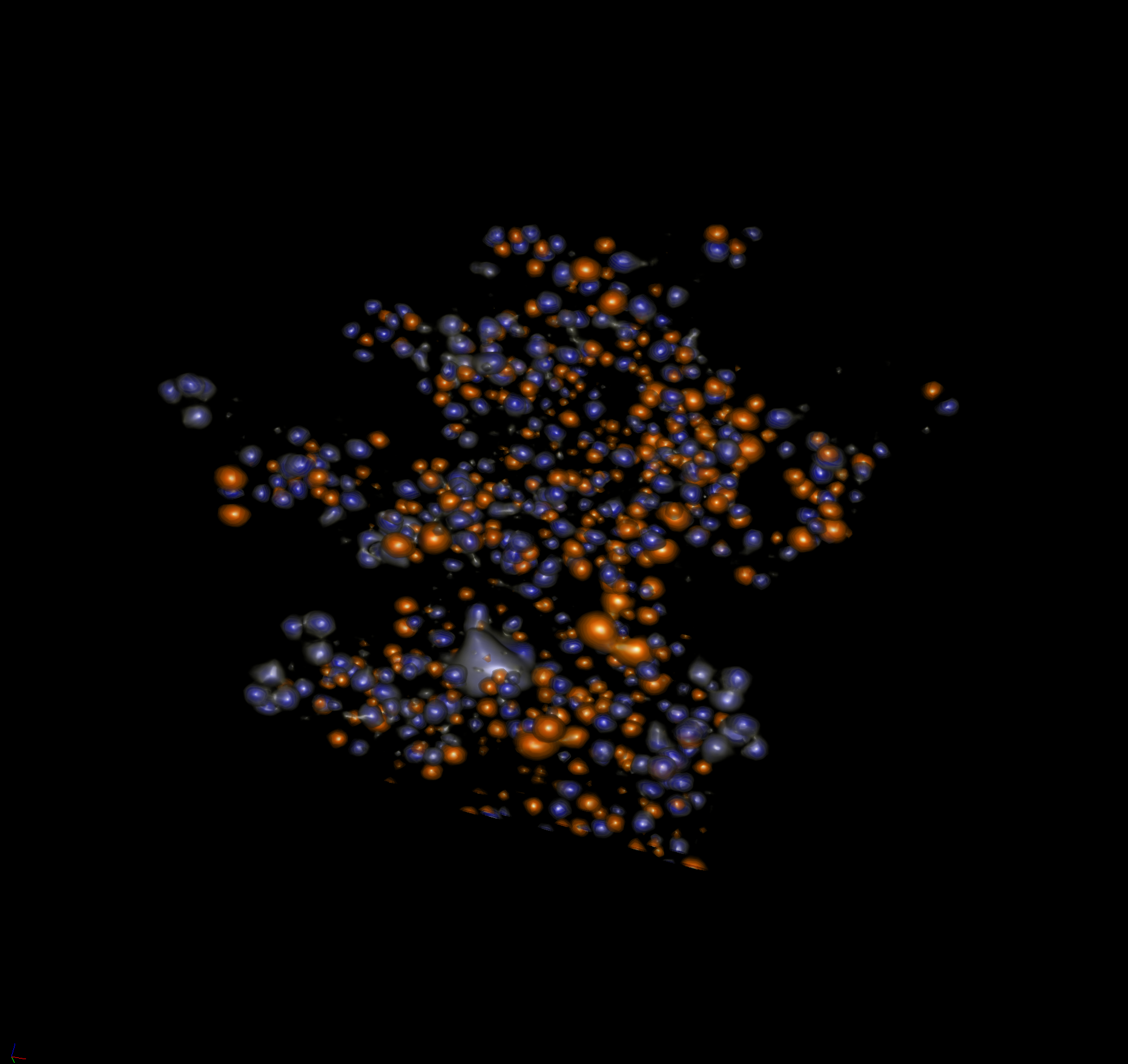}
	\caption{Electrostatic potential field obtained from the
          solution of the NPBE for the RNA binding domain of E. Coli
          RHO factor. The potential fields where created with the
          Adaptive Poisson Boltzmann Solver \cite{Baker2001} rendered
          with VolRover \cite{Bajaj2003,Bajaj2005}.  The positive and
          negative potential are rendered with blueish and
          orange/reddish colors respectively.}
	\label{nPBE:fig1}
\end{figure}

In \cite{Holst1994} Holst shows the existence and uniqueness of the
solution in the appropriate functional spaces. This approach relies on
the construction of a convex functional where the unique minimal
energy state corresponds to the solution of the nPBE.

In this paper we are interested in determining the existence and
uniqueness of the solution of the nPBE by extending the dielectric and
Debye-H\"uckel parameters into the complex domain. The convexity
theory developed in \cite{Holst1994} is no longer valid for this case,
thus motivating the construction of a novel theory to deal with the
complex case. Our ultimate goal is to study analytic extension of the
solution with respect to the complex parameters. This provides an
approach to determine the regularity and sensitivity of the solution
with respect to the dieletric and Debye-H\"uckel parameters.  This has
important connections to uncertainty quantification theory
\cite{babusk_nobile_temp_10,nobile_tempone_08,
  Castrillon2016,Castrillon2020,Castrillon2021}.

Here, we state a simplified version of our main statement; for more
details, see \cref{Schauder,Banach}.
\begin{theorem}\label{mainresult}
	Let $\Omega \subseteq \mathbb{R}^3$ be open, bounded, and
        convex with a smooth boundary. There exists $r =
        r(\epsilon,\kappa,\Omega)>0$ such that whenever $(f,g)\in
        \overline{B(0,r)} \subseteq L^2(\Omega)\times
        H^{\frac{3}{2}}(\partial\Omega)$, where \(\overline{B(0,r)}\)
        is a closed ball centered at the origin with radius $r>0$ with
        respect to the product norm, there exists a solution $u \in
        H^2(\Omega)$ to \cref{npb}. With further technical hypotheses
        on the parameters (such as $\epsilon,\kappa ^2,\Omega$), this
        solution is unique in a small ball in $H^2(\Omega)$.
\end{theorem}

Most notably, our result is an analysis of a PDE with complex-valued
functions. This renders a direct application of calculus of variation
to our problem difficult, since $\mathbb{C}$ is not ordered as in
$\mathbb{R}$. One could try to identify $\mathbb{C}$ with
$\mathbb{R}^2$ and apply the variational calculus to a system of two
real-valued equations that is equivalent to \cref{npb}; however, the
real and imaginary parts of $\sinh z = \sinh (x+iy) = \sinh x \cos y +
i \cosh x \sin y$ are not convex on $\mathbb{R}^2$ (in fact,
infinitely oscillating in $y$), and therefore the method of
variational calculus cannot be applied, at least directly. This
difficulty does not arise when the given data (or the coefficient
functions) are real-valued.

Broadly speaking, the smallness constant $r>0$ in the main result is a
reflection of our approach to the problem. Our approach is based on
the topological fixed point argument. The desired fixed point arises
as a consequence of a compact operator that we construct (see
\cref{nonlinearoperator}) where the compactness is contingent upon
choosing $r>0$ sufficiently small.

If the parameters are complex, then a solution to nPBE, even if it
exists, is generally not globally unique in the solution space, in
stark contrast to its real analogue; see \cref{holst}. We construct a
specific example in Appendix B that illustrates the existence of
multiple (smooth) solutions by posing the nPBE on a symmetric
domain. Moreover, we prove that if the technical hypotheses of
\cref{mainresult} are not satisfied then multiple non-trivial
solutions may exist for the homogeneous nPBE (see
\cref{nonunique_solutions} for more details).

This paper is organized as follows. In \cref{Preliminaries}, we
introduce useful notations and fundamental mathematical background. In
\cref{exuniq}, we state and prove the existence and uniqueness
result. In \cref{constantest}, we give estimates for the various
constants used in this paper. In \cref{appendixB}, we discuss the
failure of uniqueness of solution given a large inhomogeneous data.

\section{Preliminaries.}\label{Preliminaries}

\subsection{Sobolev Spaces}

The solution that we desire to obtain is a complex-valued function on
$\Omega$ with certain regularity and integrability, and so we briefly
recall the necessary mathematical background. In this paper, a Banach
space $X$ is assumed to be over $\mathbb{C}$ unless stated
otherwise. As is usual, the integer-order Sobolev space is defined via
weak derivatives. For $k \in \mathbb{N}\cup \{0\}$ and $\Omega
\subseteq \mathbb{R}^d$ open, bounded with a smooth boundary,
\[
	H^k(\Omega) = \{u \in L^2(\Omega): \lVert u
        \rVert_{H^k(\Omega)}<\infty\} \quad \text{with } \lVert u
        \rVert_{H^k(\Omega)} \coloneqq \Big(\sum_{|\alpha| \leq k}
        \lVert \partial_\alpha u \rVert_{L^2}^2\Big)^{\frac{1}{2}}.
\]
For $s^\prime \in [0,1)$, recall the Gagliardo seminorm
\[
	[u]_{s^\prime} \coloneqq \Big(\int_\Omega\int_\Omega
        \frac{|u(x)-u(y)|^2}{|x-y|^{d+2s^\prime}}
        dxdy\Big)^{\frac{1}{2}},
\]
by which fractional Sobolev spaces are defined. Given $s \geq 0$, we
have $s = k + s^\prime$ for $k \in \mathbb{N}\cup\{0\}$ and $s^\prime
\in [0,1)$. Define
\[
	H^s(\Omega) = \{u \in L^2(\Omega): \lVert u \rVert_{H^s} <
        \infty\}\quad \text{with } \lVert u \rVert_{H^s} \coloneqq
        \Big(\lVert u \rVert_{H^k}^2 + [u]_{s^\prime}^2
        \Big)^{\frac{1}{2}},
\]
and $H^s_0(\Omega)$ to be the closure of $C^\infty_c(\Omega)$, the
collection of smooth and compactly supported functions in $\Omega$,
under $\lVert \cdot \rVert_{H^s}$. For a more thorough discussion on
this material including the Sobolev spaces on the boundary $\partial
\Omega$ and the negative-order Sobolev spaces, see \cite[Chapter
  3]{mclean2000strongly}. For a discussion regarding the structural
difference between Banach spaces on $\mathbb{C}$ and $\mathbb{R}$, see
\cite[Chapter 11]{brezis2010functional}.

Now, we comment on the regularity of boundary data. Recall that
$g:\partial \Omega \rightarrow \mathbb{C}$ and $w\in L^2(\Omega)$ such
that $w=g$ in the trace sense. Since our proof heavily relies on the
Elliptic Regularity Theorem, we assume that $w \in H^2(\Omega)$. To
motivate this assumption, consider the linear elliptic PDE $Lu=f$ on
$\Omega$ with $u=g$ on $\partial \Omega$. Assuming that $w$ exists, a
formal calculation reveals that $\tilde{u}+w$ is the solution where
$\tilde{u}$ satisfies $L\tilde{u} = f-Lw$ on $\Omega$ with
$\tilde{u}=0$ on $\partial \Omega$. If $f \in L^2(\Omega)$, we want
$Lw \in L^2(\Omega)$ as well to ensure that $\tilde{u}$ has two more
derivatives than $f-Lw$. Another application of the Elliptic
Regularity Theorem yields $w \in H^2(\Omega)$, which in turn is
guaranteed by assuming $g \in H^{\frac{3}{2}}(\partial \Omega)$ by the
following lemma:
\begin{proposition}\cite[Theorem 3.37]{mclean2000strongly}\label{trace}
	Let $T:C^\infty(\overline{\Omega})\rightarrow C^\infty
        (\partial \Omega)$ be given by $u \mapsto u|_{\partial
          \Omega}$. If $k \in \mathbb{N}$ and $\Omega \in C^{k-1,1}$,
        then $T$ uniquely extends to a surjective bounded linear
        operator from $H^s(\Omega)$ to
        $H^{s-\frac{1}{2}}(\partial\Omega)$ for all $s \in
        (\frac{1}{2},k]$. The trace map $T$ has a right-continuous
inverse.
\end{proposition}
The map $g\mapsto w$ is not unique although one can uniquely solve the Laplace equation
\[
\begin{split}
	\Delta w&=0,\: x \in \Omega\\
	w&=g,\: x \in \partial \Omega,\nonumber
\end{split}
\]
and obtain an explicit form for $w$ as an integration against the
Poisson kernel, and thereby establish a map $T^{-1}g\coloneqq w$. It
can be shown that $T^{-1}: H^{k-\frac{1}{2}}(\partial
\Omega)\rightarrow H^k(\Omega)$ defines a bounded linear operator
where the operator norm depends on $k\in \mathbb{N}$ and $\Omega$. In
this paper, we are not concerned with this constant; henceforth, given
$g \in H^{\frac{3}{2}}(\partial \Omega)$, we fix $w \in H^2(\Omega)$
and work entirely with functions defined on the domain, not the
boundary.

\subsection{Principal Eigenvalue of the Dirichlet Laplacian}

Throughout this paper, we let $\Omega \subseteq \mathbb{R}^d$ be open,
bounded, convex, and connected with a smooth boundary. Let $|\Omega|$
denote the Lebesgue measure of $\Omega$ and $d_\Omega \coloneqq
\sup_{x,y\in \Omega}|x-y|$, the diameter of $\Omega$. Let
$\left\{\lambda_i\right\}_{i=1}^\infty$ be the eigenvalues of the
(negative) Dirichlet Laplacian $-\Delta$, the Laplacian operator
restricted to functions vanishing on the boundary defined via the
Friedrich extension, where they are ordered such that
\[
	0<\lambda_1<\lambda_2\leq \lambda_3\leq...
\]
The principal eigenvalue is given by
\[
	\lambda_1 = \min_{u \in H^1_0(\Omega)\setminus \{0\}}
        \frac{\int_{\Omega} |\nabla u|^2}{\int_\Omega |u|^2}.
\]
The variational formula above directly implies the Poincar\'e
inequality given below:
\begin{equation}\label{poincare2}
	\lVert u \rVert_{L^2(\Omega)}^2\leq \lambda_1^{-1}\lVert
        \nabla u \rVert_{L^2(\Omega)}^2, \quad \forall u \in
        H^1_0(\Omega).
\end{equation}
\begin{remark}\label{poincaresharp}
	If we further assume that $\Omega$ is convex, then $\lambda_1
        \geq \frac{\pi^2}{d_\Omega^2}$ by
        \cite{payne1960optimal}. Conversely, it can be shown without
        convexity that $\lambda_1 \leq \frac{c}{d_\Omega^2}$ for some
        constant $c=c(\Omega,d)>0$, and therefore $\lambda_1$ is
        bounded above and below by $\frac{1}{d_\Omega^2}$ up to a
        constant.
	
	Now, we justify the last claim. By translation, assume $0\in
        \Omega$. There exists $R>0$ such that $B(0,R) \subseteq
        \Omega$. Denote $U = B(0,R)$ and $U^\prime =
        B(0,\frac{R}{2})$. Let $c_1>0$ such that $R=c_1
        d_\Omega$. Construct a smooth function
        $v:[0,\infty)\rightarrow [0,1]$ such that $v=1$ on
          $[0,\frac{R}{2}]$, $v=0$ on $[R,\infty)$, and $0<v<1$ on
            $(\frac{R}{2},R)$ such that $\sup\limits_{r \in
              [0,\infty)}|v^\prime(r)| = \frac{c_2}{R}$ for some
              $c_2>0$. Define $u(x) = v(|x|)$. Then,
	\[
		\lVert u \rVert_{L^2(\Omega)}^2 = \lVert u
                \rVert_{L^2(U)}^2 \geq |U|^{-1} \lVert u
                \rVert_{L^1(U)}^2 \geq |U|^{-1} |U^\prime|^2 = 2^{-2d}
                |U| = \frac{2^{-2d}|S^{d-1}|R^d}{d},
	\]
	where $S^{d-1}$ is the unit sphere in $\mathbb{R}^d$. On the
        other hand,
	\[
		\int_\Omega |\nabla u|^2 =|S^{d-1}| \int_0^R
                |v^\prime(r)|^2 r^{d-1} dr \leq \frac{c_2^2
                  |S^{d-1}|}{d}R^{d-2},
	\]
	and hence
	\[
		\lambda_1 \leq \frac{\int_\Omega |\nabla
                  u|^2}{\int_\Omega |u|^2} \leq
                2^{2d}\Big(\frac{c_2}{c_1}\Big)^2 d_\Omega^{-2}.
	\]
\end{remark}

\section{Existence and uniqueness.}\label{exuniq}

\subsection{Sketch of Main Results and Assumptions}\label{sketch_assumptions}

To show that a unique solution to \cref{npb} exists we leverage the
Schauder's fixed point theorem. This method is developed by first
rewriting \cref{npb} as the functional equation
\begin{equation}\label{functional_equation}
    F(u) = f
\end{equation}
with \(u\in H^2(\Omega)\) and \(f\in L^2(\Omega)\) and \(F:H^2(\Omega)
\to L^2(\Omega)\) a nonlinear map satisfying \(F(0) = 0\). The map
\(F\) can be decomposed into its linear and nonlinear components. That
is we can take \(F(u) = Lu + N(u)\), where \(L := DF(0)\) and \(N(u)
:= F(u) - Lu\). Assuming that \(L\) is invertible, then
\cref{functional_equation} can be re-expressed as
\[
   u = L^{-1}( f - N(u))
\]
with possible additional terms added to account for boundary
conditions. So our solution \(u\) should be a fixed point of the
mapping \(u \mapsto L^{-1}(f - N(u))\). The typical approach to
finding a fixed point is to demonstrate the map is a contraction and
apply the Banach Fixed Point Theorem to get a unique
solution. However, since \(N(u)\) can increase quadratically in norm
as \(u\) becomes larger in norm (in our case \(\| N(u) \|\) can grow
\emph{exponentially}), it is not possible for us to have this map be a
contraction on \(H^2(\Omega)\). Instead, we restrict the norm of \(u\)
to a smaller space, add additional hypotheses to the parameters of
\cref{npb}, and apply a corollary of Schauder's fixed point theorem to
get existence of a solution.

To apply this argument to \cref{npb} we make the following assumptions
on the parameters of the PDE:

\begin{hypothesis}\label{h1}
	The domain \(\Omega\subset \mathbb R^d\) is open, bounded, and
        convex with smooth (at least \(C^2\)) boundary.
\end{hypothesis}
\begin{hypothesis}\label{h2}
	The function \(\epsilon \in W^{1,\infty}(\Omega, \mathbb
        C^{d^2})\), where \(\epsilon^{ij} = \epsilon^{ji}\), satisfies
        the following uniform ellipticity condition: there exists
        \(\theta > 0\) such that
	\begin{equation}\label{ellipticity}
		\Real\left[ \sum_{i,j=1}^d \epsilon^{ij}(x) \xi_i
                  \overline{\xi_j} \right] \geq \theta |\xi|^2
	\end{equation}
	for all \(\xi\in \mathbb{C}^d\) and for a.e.\ \(x\in \Omega\).
\end{hypothesis}
\begin{hypothesis}\label{h3}
	There exists \(\mu \geq 0\) such that \(\kappa^2 \in
        L^\infty(\Omega)\) satisfies
	\begin{equation}\label{kappa_bound}
		\Real\left[\mathrm{\kappa^2(x)}\right] \geq -\mu
	\end{equation}
	for a.e.\ \(x\in \Omega\). Moreover, the parameters \(\theta\)
        and \(\mu\) characterized by \cref{ellipticity,kappa_bound},
        respectively, satisfy the inequality
	\begin{equation}
		\frac\mu\theta < \lambda_1,
	\end{equation}
	where \(\lambda_1\) is the principal eigenvalue of \(-\Delta\) on \(\Omega\).
\end{hypothesis}

\Cref{h2,h3} are sufficient to guarantee a unique strong solution
within a small ball in \(H^2(\Omega)\), and omitting these hypotheses
may result in non-uniqueness. If the parameters are allowed to
continuously vary until \cref{h2,h3} no longer hold, then there may be
a bifurcation of the unique small solution. For the case where
\(\epsilon\) and \(\kappa^2\) are scalar-valued and \(f\) and \(g\)
set to zero functions, \cref{npb} simplifies to

\begin{equation}\label{scalar_NPBE}
	\begin{aligned}
		-\Delta u + (\eta - \lambda_1) \sinh(u) &= 0 \quad
                \text{ in } \Omega, \\ u & = 0 \quad \text{ on }
                \partial \Omega,
	\end{aligned}
\end{equation}
where \(\eta = \lambda_1 + \kappa^2/\epsilon\in \mathbb{R}\). The
function \(u\equiv 0\) is always a trivial solution for
\cref{scalar_NPBE}. For \cref{scalar_NPBE}, \cref{h2,h3} are satisfied
if and only if \(\eta\) is greater than zero. The parameter \(\eta\)
is the smallest eigenvalue of the \(-\Delta + (\eta-\lambda_1)\), so
this linear operator is non-invertible when \(\eta = 0\). The
following result proved by Crandall and Rabinowitz in
\cite{crandall1971bifurcation} can be used to show that the zero
solution undergoes a bifurcation at \(\eta=0\):

\begin{theorem}\label{CR_theorem}
	Let \(X\) and \(Y\) be Banach spaces and assume
	\begin{enumerate}[label= (\roman*)]
		\item \(F \in C^2(X\times \mathbb R, Y)\),
		\item \(F(0,\eta) = 0\) for all \(\eta \in \mathbb R\),
		\item \(\dim{N(D_xF(0,0))} = \mathrm{codim}{R(D_x F(0,0) )} = 1\), and 
		\item \(D^2_{x\eta} F(0,0) \hat{v}_0 \notin R(D_xF(0,0))\) where \(\hat{v}_0\neq 0\)  is in  \(N(D_xF(0,0))\). \label{item_4}
	\end{enumerate}
	Then there is a nontrivial continuously differentiable curve
	\begin{equation}
		\{(x(s),\eta(s)) \mid s\in (-\delta, \delta) \}
	\end{equation}
	such that \((x(0), \eta(0)) = (0,0)\), \(x'(0) = \hat{v}_0\) and 
	\begin{equation}
		F(x(s), \eta(s)) = 0 \quad \text{ for } \quad s\in(-\delta, \delta).
	\end{equation}
\end{theorem}

Here \(D_x\) and \(D_\eta\) represent the Frechet derivatives of \(F\)
with respect to the \(X\) and \(\mathbb R\) components,
respectively. Note that \(D_\eta F(x,\eta) \in \mathcal L (\mathbb R,
Y)\), the set of linear operators from \(\mathbb R\) into \(Y\). An
element \(A \in \mathcal L (\mathbb R,Y)\) can be uniquely associated
with an element \(y\in Y\) by setting \(y\in A(1)\). Thus
\(D_{x\eta}^2F(x,\eta)\) can be associated with an element of
\(\mathcal L(X,Y)\), which is how the map \(D^2_{x\eta}F(0,0)\) is
being viewed in \cref{item_4}. Applying \cref{CR_theorem} gives the
existence of non-unique small solutions.

\begin{theorem}\label{nonunique_solutions}
	Let \(c > 0\). Then there is \(\eta^* < 0\) such that for any
        \(\eta \in [\eta^*, 0)\) there is a non-trivial solution \(u\)
          of \cref{scalar_NPBE} such that \(\|u\|_{H^2} < c\).
\end{theorem}

\begin{proof}
	Define \(F:H^2(\Omega)\cap H^1_0(\Omega) \times \mathbb R \to L^2(\Omega)\) as 
	\begin{equation}
		F(u,\eta) = -\Delta u + (\eta - \lambda_1) \sinh(u).
	\end{equation}
	Assumptions \((i)\) and \((ii)\) are clearly satisfied. Note
        that \(D_u F(0,0) = -\Delta - \lambda_1\). Thus assumption
        \((iii)\) follows from the Fredholm properties of \(-\Delta\)
        and from the fact that \(\lambda_1\) is the principal
        eigenvalue of \(-\Delta\). Let \(\hat{v}_0\) be be the
        eigenfunction corresponding to \(\lambda_1\). Then assumption
        \((iv)\) states that \(D_{u\eta}^2 F(0,0) \hat{v}_0 =
        \hat{v}_0 \notin R(D_u F(0,0))\). This should hold since if
        \(\hat{v}_0 \in R(D_u F(0,0))\), then there is a generalized
        eigenfunction for \(\lambda_1\) which contradicts the
        simplicity of \(\lambda_1\). Hence, we can apply
        \cref{CR_theorem} to \(F(u,\eta) = 0\) to get a nontrivial
        continuously differentiable curve
	\begin{equation}\label{curve}
		\{(u(s), \eta(s)) \mid s\in(-\delta,\delta)\}.
	\end{equation}
	with \((u(0), \eta(0)) = (0,0)\).
	
	One can also compute derivatives of \(\eta(s)\) (see \cite[\S
          1.6]{kielhofer2011bifurcation} for details) to get that
        \(\eta'(0) = 0\) and \(\eta''(0) < 0\). Thus it is possible to
        choose \(\delta\) small enough so that \(\eta(s)\) is strictly
        decreasing on \(s\in[0,\delta)\). Given \(c>0\), we can find
          \(s^*>0\) small enough to get \(\|u(s)\|_{H^2} \leq c\) for
          all \(s\in (0, s^*]\) by the continuity of the curve. Since
        \(s\mapsto \eta(s)\) is strictly decreasing on \((0,s^*]\), we
        can invert this mapping to get \(\eta\mapsto s(\eta)\) for
        \(\eta\in [\eta^*,0)\) where \(\eta^* = \eta(s^*).\)
          Therefore, for each \(\eta\in [\eta^*,0)\) we have a
            solution of \cref{scalar_NPBE}, given by \(u =
            u(s(\eta))\) such that \(\|u\|_{H^2}<c\).
\end{proof}

\begin{remark}
	There are two non-trivial solutions to \cref{scalar_NPBE} for
        \(\eta< 0\) sufficiently close to zero: one for \(s>0\) and
        one for \(s<0\). In fact, from the oddness of \(\sinh(u)\), if
        \(u\) is a solution of \cref{scalar_NPBE} then so is \(-u\).
\end{remark}

Comparing \cref{nonunique_solutions} with \cref{Banach} shows that
uniqueness cannot be guaranteed without \cref{h2,h3}.  Moreover, this
can also be true for the case of the non-homogeneous NPBE as
demonstrated in \cref{appendixB}.

\subsection{Definitions}

Before showing existence and uniqueness of solutions, we must first
make precise how a function is a weak or strong solution of
\cref{npb}.

\begin{definition}\label{npbweak}
	A function $u \in H^1(\Omega)$ is a \emph{weak solution} to
        \cref{npb} if for all $\phi \in H^1_0(\Omega)$, we have
	\begin{equation}
	\begin{split}
	\label{npbweak2}
		\int_\Omega (\epsilon\nabla u) \cdot \overline{\nabla
                  \phi} + \int_\Omega \kappa^2\sinh u
                \cdot\overline{\phi} &=\int_\Omega f
                \overline{\phi}\\ u|_{\partial\Omega}&=g,
		\end{split}
	\end{equation}
	where the equality at the boundary is in the trace sense; if
        $w \in H^1(\Omega)$ whose trace is $g \in
        H^{\frac{1}{2}}(\partial \Omega)$, a weak solution $u$
        satisfies $u-w \in H^1_0(\Omega)$. If a weak solution $u$ is
        twice weakly-differentiable and satisfies \cref{npb} pointwise
        almost everywhere, then we say $u$ is a \emph{strong
          solution}. We say $u\in H^1(\Omega)$ is a weak solution to
        the linearized Poisson-Boltzmann equation if it satisfies
        \cref{npbweak2} with $\sinh u$ replaced by $u$. We say that
        the nPBE equation is homogeneous if $(f,g) = (0,0)$.
\end{definition}

A weak solution that is in $H^2(\Omega)$ satisfies the strong form
a.e.\ if one can \textit{rewind} the integration by parts in the first
term of \cref{npbweak2}. This is possible since \(\epsilon\in
W^{1,\infty}(\Omega, \mathbb{C}^{d^2})\) . If $u \in H^2(\Omega)$ is a
weak solution, then one can undo the integration by parts since
$\epsilon \nabla u \in H^1(\Omega)$. Indeed for each $1 \leq i,k \leq
d$,
\[
	\lVert \partial_k \sum_{j = 1}^{d} (\epsilon^{ij}\partial_j
        u)\rVert_{L^2} \leq C \lVert \epsilon
        \rVert_{W^{1,\infty}}\lVert u \rVert_{H^2}.
\]

Setting up the functional equation given in \cref{sketch_assumptions},
the non-linear operator \(F:H^2(\Omega)\to L^2(\Omega)\) is given by
\begin{equation*}
    F(u) = -\nabla\cdot(\epsilon \nabla u) + \kappa^2 \sinh(u)
\end{equation*}
so that \(u\in H^2(\Omega)\) is a strong solution to the \cref{npb} if
it satisfies
\begin{equation*}
\begin{aligned}
	F(u) &= f, & x&\in \Omega\\
	u&= g, &  x &\in \partial \Omega.
\end{aligned}
\end{equation*} 
The function \(F\) can be broken up into its linear and non-linear
components so that $F(u) = Lu + N(u)$ with
\begin{align}
	Lu &= -\nabla\cdot(\epsilon\nabla u) + \kappa^2 u \label{set-up} \\
	N(u) &= \kappa^2\sinh(u) = \kappa^2 \sum_{k=2}^\infty n_{k} u^{k},
\end{align}
where $L$ is a linear uniformly elliptic second-order differential
operator by \cref{h2} and $n_{k} = \frac{1}{(k)!}$ for odd \(k\) and
$n_k = 0$ for even \(k\). Note that when \(d\leq 3\), \(H^2(\Omega)\)
is an algebra and so \(u^{k}\in H^2(\Omega)\) and \(N(u)\) is defined.

In order to determine if a strong solution to the NPBE exists, we also
need control of certain fundamental constants of the operators and
Sobolev spaces defined. To simplify the presentation and discussion in
this section we define the relevant constants $C_S(s)$, $C_D$ and
$C_H$.  Explicit estimates for these constants and their dependencies
are derived in detail in Appendix A.

\begin{definition}\label{important_constants}
The Sobolev embedding for $d=3$ and $s>\frac{3}{2}$ gives that
\(H^s(\Omega) \hookrightarrow L^\infty(\Omega)\) \cite[Theorem
  3.26]{mclean2000strongly}. The constant $C_S = C_S(s,\Omega)>0$ is
then defined as the norm of the inclusion operator from
\(H^s(\Omega)\) into \(L^\infty(\Omega)\):
\begin{equation*}
    C_S(s,\Omega) = C_S(s):= \inf\{ C>0 : \|u\|_{L^\infty(\Omega)} \leq C \|u\|_{H^s(\Omega)},\ \forall u \in H^s(\Omega)\}.
\end{equation*}
The linear operator $L$ in \cref{set-up} takes functions in
\(H^2(\Omega)\cap H^1_0(\Omega)\) to functions in \(L^2(\Omega)\) We
shall denote \(C_D>0\) to be the norm of \(L\) with respect to these
spaces:
\begin{equation*}
    C_D := \inf\{ C>0 : \|Lu\|_{L^2(\Omega)} \leq C
    \|u\|_{H^2(\Omega)},\ \forall u \in H^2(\Omega)\}.
\end{equation*}
By \cref{h3}, \(L\) is invertible and so \(L^{-1}:L^2(\Omega) \to
H^2(\Omega) \cap H^1_0(\Omega)\) is defined. Define \(C_H>0\) to be
the norm of \(L^{-1}\):
\begin{equation*}
    C_H := \inf\{ C>0 : \|u\|_{H^2(\Omega)} \leq C
    \|Lu\|_{L^2(\Omega)},\ \forall u \in H^2(\Omega)\}.
\end{equation*}
\end{definition}
\subsection{Main Results}

In this section we present the main ideas of our paper.  We employ a
fix point argument to show the existence of the solution to the
nPBE. This is based on producing a sequence of solutions of linear
elliptic PDEs and showing that such sequence converges to the solution
of the nPBE. We first state two useful lemmas that will be crucial to
our mathematical argument.

\begin{lemma}[Lax-Milgram Theorem ]\cite[Chapter 6]{evans2010partial}\label{LM}
	Given a complex Hilbert space $\mathscr{H}$, let
        $B:\mathscr{H}\times\mathscr{H}\rightarrow\mathbb{C}$ be a
        sesquilinear form that is bounded and coercive. By coercivity,
        suppose there exists $\beta>0$ such that $\Real(B(u,u))\geq
        \beta \lVert u \rVert^2$ for all $u \in \mathscr{H}$. Then,
        for every $f \in \mathscr{H}^\prime$, there exists a unique $u
        \in \mathscr{H}$ such that $B(u,\phi) = \langle f,\phi
        \rangle$ for all $\phi \in \mathscr{H}$ where $\langle
        \cdot,\cdot\rangle$ denotes the dual pairing. Moreover, we
        have $\lVert u \rVert \leq \beta^{-1} \lVert f
        \rVert_{\mathscr{H}^\prime}$.
\end{lemma}

For the readers' convenience, we also state fixed point theorems that
we need later.
\begin{lemma}\label{fixedpointthm}
	Let $X$ be a Banach space over $\mathbb{R}$ or $\mathbb{C}$,
        and let $A:X \rightarrow X$.
	\begin{enumerate}
		\item (Banach's Fixed Point Theorem) If there exists
                  $\gamma \in (0,1)$ such that $\lVert A[u] -
                  A[\tilde{u}]\rVert \leq \gamma \lVert u - \tilde{u}
                  \rVert$ for all $u,\tilde{u} \in X$, then there
                  exists a unique $u_0 \in X$ such that $A[u_0]=u_0$.
		\item \cite[Corollary 11.2]{gilbarg2015elliptic} Let
                  $K \subseteq X$ be closed and convex. If
                  $A:K\rightarrow K$ and $\{A[u]: u \in K\}$ is
                  precompact, then there exists $u_0 \in K$ such that
                  $A[u_0] = u_0$.
	\end{enumerate}
\end{lemma}

We now show that under \cref{h1,h2,h3} there exists a unique solution
to the linear PBE. Note that in the rest of the discussion in this
section we assume that \cref{h1,h2,h3} are always true.

\begin{lemma}\label{laxmilgram}
	Let $L$ be as in \cref{set-up} and let $f\in L^2(\Omega)$. Fix
        $w \in H^2(\Omega)$ whose trace is $g \in
        H^{\frac{3}{2}}(\partial\Omega)$. Then, there exists a unique
        $u \in H^1(\Omega)$ such that $u = g$ in the trace sense and
	\[
		\int_\Omega (\epsilon \nabla u)\cdot
                \overline{\nabla\phi} + \kappa^2 u \overline{\phi} =
                \int_\Omega f\overline{\phi},
	\]
	for all $\phi \in H^1_0(\Omega)$. 
\end{lemma}
\begin{proof}
	Define a sesquilinear form $B:H^1_0(\Omega)\times
        H^1_0(\Omega)\rightarrow\mathbb{C}$ given by
	\[
		B(u,\phi) = \int_\Omega (\epsilon \nabla u)\cdot
                \overline{\nabla\phi} + \kappa^2 u \overline{\phi},
	\]
	which defines a bounded operator. Additionally, $B$ is
        coercive since
	\begin{equation}\label{LM2}
	\begin{split}
		\Real[B(u,u)] &= \int \Real[(\epsilon \nabla
                  u)\cdot\overline{\nabla u}] + \int
                \Real(\kappa^2)|u|^2 \\ &\geq \theta \int |\nabla u|^2
                - \mu \int |u|^2\geq (\theta - \mu \lambda_1^{-1})\int
                |\nabla u|^2,
		\end{split}
	\end{equation}
	where the last inequality is by \cref{poincare2}. Hence, there
        exists a unique $\tilde{u}\in H^1_0(\Omega)$ that satisfies
	\[
		B(\tilde{u},\phi) = \int (f-Lw)\overline{\phi},
	\]
	for all $\phi \in H^1_0(\Omega)$ by \cref{LM}. The proof is
        complete by taking $u \coloneqq \tilde{u}+w$.
\end{proof}

A key ingredient in our approach is elliptic regularity. Let $L$ be as
in \cref{set-up} and consider solving $Lu=f \in L^2(\Omega)$ with
$g=0$ on $\partial \Omega$. From the standard elliptic theory (for
instance, see \cite[Section 6.3, Theorem 4]{evans2010partial}), there
exists $C_H>0$ as described in \cref{important_constants} such that
$\lVert u \rVert_{H^2} \leq C_H \lVert f \rVert_{L^2}$. For general
boundary data, we have

\begin{lemma}\label{ellipticregularity}
	The unique weak solution of the linear PBE, $u \in
        H^1(\Omega)$, satisfies
	\begin{equation}\label{ellipticregularity2}
		\lVert u \rVert_{H^2} \leq C_H \lVert f \rVert_{L^2} +
                (C_HC_D + 1)\lVert w \rVert_{H^2}.
	\end{equation}
\end{lemma}
\begin{proof}
	As in the proof of \cref{laxmilgram}, absorb the boundary data
        into the inhomogeneous term by replacing $f$ by $f-Lw$ and
        considering the zero Dirichlet boundary condition. Then, an
        application of the Elliptic Regularity Theorem yields the
        estimate \cref{ellipticregularity2}.
\end{proof}

To estimate the non-linear term in $L^2(\Omega)$, we work with
functions with sufficiently high Sobolev regularity.

\begin{lemma}
	Let $s>\frac{d}{2}$. Then for every $u\in H^s(\Omega)$,
	\begin{equation}\label{gn}
		\lVert N(u) \rVert_{L^2} \leq \lVert \kappa^2
                \rVert_{L^\infty}|\Omega|^\frac{1}{2}\sum_{k=2}^\infty
                |n_k|(C_S(s) \lVert u \rVert_{H^s})^k.
	\end{equation}
	For $N(u)$ for \cref{npb}, we have
	\[
		\lVert N(u)\rVert_{L^2} \leq \lVert \kappa^2
                \rVert_{L^\infty}|\Omega|^{\frac{1}{2}}\Big(\sinh(C_S(s)
                \lVert u \rVert_{H^s}) - C_S(s) \lVert u
                \rVert_{H^s}\Big).
	\]
\end{lemma}

\begin{proof}
	\[
		\lVert N(u)\rVert_{L^2} \leq \lVert \kappa^2
                \rVert_{L^\infty}\sum_{k=2}^\infty |n_k| \lVert u^k
                \rVert_{L^2} \leq \lVert \kappa^2
                \rVert_{L^\infty}|\Omega|^{\frac{1}{2}}\sum_{k=2}^\infty
                |n_k| \lVert u\rVert_{L^\infty}^k\leq \lVert \kappa^2
                \rVert_{L^\infty}|\Omega|^\frac{1}{2}\sum_{k=2}^\infty
                |n_k|(C_S(s) \lVert u \rVert_{H^s})^k,
	\]
	where the inequalities are by the triangle inequality, the
        H\"older's inequality, and the Sobolev inequality,
        respectively.
\end{proof}

\begin{remark}
	By working in Sobolev algebras, we bypass the problem of
        whether $N(u) \in L^2(\Omega)$ or not, for $u \in
        L^\infty(\Omega)$. Note that Holst \cite[Chapter 2]{Holst1994}
        bypasses this issue as well, not by working with more regular
        functions as we do, but by constructing a conditional action
        functional on $H^1_0(\Omega)$. Such an approach does not apply
        in our setting where complex-valued functions are studied.
	
	For $d=3$, we have the embedding $H^2(\Omega)\hookrightarrow
        L^\infty(\Omega)$, but such embedding for $H^1(\Omega)$ does
        not hold. For $d \geq 3$, it is straightforward to construct
        an example of $u \in H^1(\Omega)$ such that $N(u) = \kappa^2
        (\sinh u -u) \notin L^2(\Omega)$. For simplicity, take $\Omega
        = \mathbb{R}^d$ and $\kappa=1$. For $R>0$, define $u(x) =
        |x|^{-\alpha}\zeta(x)$ where $\alpha = \frac{d}{2}-1-\epsilon$
        with $\epsilon \ll 1$ and $\zeta \in C^\infty_c(B(0,R))$ is a
        smooth non-negative function such that $\zeta = 1$ on
        $\overline{B(0,\frac{R}{2})}$. If $N(u) \in L^2(\Omega)$, then
        $\sinh (u(\cdot)) \in L^2(\Omega)$. Since $\sinh u \geq
        \frac{u^N}{N!}$ for every odd $N \geq 1$, we have $u^N \in
        L^2(\Omega)$. However, this is false due to the blow-up of $u$
        at the origin.
\end{remark}
To prove the existence of a solution, define a non-linear operator $A:
C^\infty_c(\Omega)\rightarrow H^2(\Omega)$ where for every $u \in
C^\infty_c(\Omega)$, $A(u)$ satisfies
\begin{equation}
\label{nonlinearoperator}
\begin{split}
	L(A(u)) &= f-N(u),\:x\in \Omega\\
	A(u)&= g,\: x \in \partial \Omega.
	\end{split}
\end{equation}
Denoting $K: L^2(\Omega)\rightarrow H^1_0(\Omega)\cap H^2(\Omega)$ to
be the inverse of $L$ stated in \cref{ellipticregularity}, we conclude
\[
	A(u) = K(f-N(u)-Lw)+w.
\]

Equivalently, the operator $A$ defines an iteration map on some Banach
space where each iterate is a unique solution to the linear PBE. More
precisely, let $u_0 \in H^2(\Omega)$ be the solution for the
linearized nPBE; by \cref{laxmilgram}, there exists a unique weak
solution $u_0 \in H^1(\Omega)$ and by \cref{ellipticregularity}, $u_0
\in H^2(\Omega)$. By the Sobolev embedding theorem,
$H^2(\Omega)\hookrightarrow C^{0,\frac{1}{2}}(\overline{\Omega})$, and
therefore $N(u_0)\in L^2(\Omega)$. Then, consider
\begin{equation}\label{approxsol}
\begin{split}
	Lu_k + N(u_{k-1}) &= f,\: k \geq 1,\\
	u_k &= g,
	\end{split}
\end{equation}
or equivalently, $u_k = A(u_{k-1})$.

There are three possible outcomes. Firstly, the sequence is divergent,
a dead end. Secondly, the sequence is convergent to a function that
does not solve \cref{npb} in any meaningful way; see
\cite{brezis2007nonlinear}. In contrast to the first two situations,
we show that $\left\{u_k\right\}$ converges to a solution of
\cref{npb}. We show that $A$ uniquely extends to a compact operator on
fractional Sobolev spaces. In the rest of this section, we assume
$d=3$.

\begin{proposition}\label{nonlinearcompact}
	$A:L^\infty(\Omega)\rightarrow H^2(\Omega)$ is
  continuous. Consequently, $A$ is continuous on $H^s(\Omega)$ for
  every $s \in (\frac{3}{2},2]$. Furthermore, $A$ is compact on
          $H^s(\Omega)$ for every $s\in (\frac{3}{2},2)$.
\end{proposition}
\begin{proof}
	Let $u_n\xrightarrow[n\rightarrow\infty]{}u$ in
        $L^\infty(\Omega)$ where $u_n,u \in L^\infty(\Omega)$. There
        exists $N\in\mathbb{N}$ such that if $n \geq N$, then $\lVert
        u_n \rVert_{L^\infty} \leq 2 \lVert u
        \rVert_{L^\infty}$. Then,
	\begin{equation}\label{continuityest}
	\begin{split}
		\lVert A(u_n)-A(u)\rVert_{H^2} &= \lVert K(N(u_n)-N(u))\rVert_{H^2} \leq C_H \lVert N(u_n)-N(u)\rVert_{L^2} \\
		&\leq C_H \lVert \kappa^2 \rVert_{L^\infty}\sum_{k=2}^\infty |n_k|\lVert u_n^k - u^k\rVert_{L^2} \\
		&\leq C_H \lVert \kappa^2 \rVert_{L^\infty}\lVert u_n-u\rVert_{L^\infty} \sum_{k=2}^\infty |n_k|\sum_{j=0}^{k-1} \lVert u_n^{k-1-j}u^j\rVert_{L^2} \\
		&\leq C_H \lVert \kappa^2 \rVert_{L^\infty}|\Omega|^{\frac{1}{2}}\lVert u_n-u\rVert_{L^\infty} \sum_{k=2}^\infty |n_k| \sum_{j=0}^{k-1} \lVert u_n \rVert_{L^\infty}^{k-1-j} \lVert u \rVert_{L^\infty}^j \\
		&\leq C_H \lVert \kappa^2 \rVert_{L^\infty}|\Omega|^{\frac{1}{2}}\lVert u_n-u\rVert_{L^\infty} \sum_{k=2}^\infty |n_k| \lVert u \rVert_{L^\infty}^{k-1} (2^k-1).
		\end{split}
	\end{equation}
	The proof is done if $u = 0$, so assume $u \neq 0$ and let $R
        > 4\lVert u \rVert_{L^\infty}$. If we show that the series in
        \cref{continuityest} converges, then the proof is complete. By
        the Cauchy integral formula, we obtain an upper bound on
        $|n_k|$
	\[
		|n_k| \leq \frac{\max\limits_{|z|=R}|N(z)|}{R^{k}},
	\]
	and combining this bound with \cref{continuityest}, the
        infinite sum is a convergent geometric series, and therefore
        the desired continuity has been shown.
	
	The rest follows from the Sobolev embedding and the
        Rellich-Kondrachov compactness theorem. Indeed, the continuity
        of $A$ on $H^s(\Omega)$ follows by considering the embedding
        $H^s(\Omega)\hookrightarrow L^\infty(\Omega)$. Compactness of
        $A$ on $H^s(\Omega)$ is immediate once we show $A$ sends a
        bounded subset of $H^s(\Omega)$ into a bounded subset of
        $H^2(\Omega)$. For $\lVert u \rVert_{H^s}\leq M$,
	\begin{equation}\label{brouwer0}
	\begin{split}
		\lVert A(u)\rVert_{H^2} &= \lVert K(f-N(u)-Lw)+w\rVert_{H^2} \leq C_H \lVert f-N(u)-Lw \rVert_{L^2}+\lVert w \rVert_{H^2} \\
		&\leq C_H \lVert f \rVert_{L^2} + (C_HC_D+1)\lVert w \rVert_{H^2} + C_H \lVert N(u)\rVert_{L^2} \\
		&\leq C_H \lVert f \rVert_{L^2} + (C_HC_D+1)\lVert w \rVert_{H^2} + C_H \lVert \kappa^2 \rVert_{L^\infty}  |\Omega|^\frac{1}{2}\sum_{k=2}^\infty |n_k|(C_S M)^k, 
	\end{split}
	\end{equation}
	where \cref{brouwer0} follows from \cref{gn}.
\end{proof}

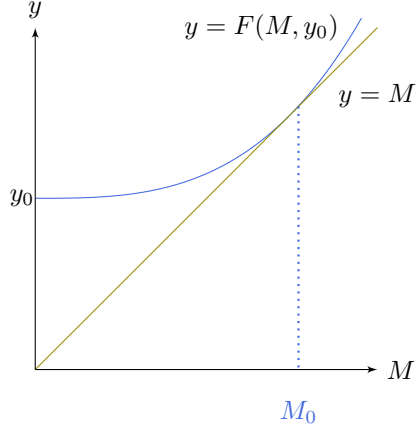
\begin{figure}[ht]
		\centering
		\newcounter{j} %
		\begin{tikzpicture}  
			[>=latex',scale=7, declare function={%
				p(\t)= greater(\t,0.5)  ? 1 : 
				0.3256 + sinh( 2 * \t ) - 2 * \t;} ]
			\draw[color=RoyalBlue,samples at={0,0.01,...,0.62}] plot (\x,{0.3256 + sinh( 2 * \x ) - 2 * \x});  
			\draw[color=olive](0,0)--(0.65,0.65);
			\draw[->](0,0)--(0,0.65) node[above]{$y$};
			\draw[->](0,0)--(0.65,0) node[right]{$M$};

			\draw[color=RoyalBlue,dotted,line width=0.8pt]%
			(0.5,0.5)--(0.5,0) node[below=8pt]{$M_0$};
			
			\node at (0.43,0.65) {$y = F(M,y_0)$};
			\node at (0.65,0.52) {$y = M$};
			\node at (-0.025,0.3256) {$y_0$};
		\end{tikzpicture}
		\caption{The tangency condition of \cref{small} where
                  $y_0 = C_H \lVert f \rVert_{L^2} + (C_HC_D+1)\lVert
                  w \rVert_{H^2}$ and $F$ is the LHS of
                  \cref{small}}\label{figure}
\end{figure} 

Now we apply the a priori estimate above to obtain a strong solution
to \cref{npb}.

\begin{theorem}\label{Schauder}
	Let $s\in (\frac{3}{2},2)$ and $M>0$ satisfy
	\begin{equation}\label{Schauder2}
		C_H \lVert f \rVert_{L^2} + (C_HC_D+1)\lVert w
                \rVert_{H^2} + C_H \lVert \kappa^2 \rVert_{L^\infty}
                |\Omega|^\frac{1}{2}\sum_{k=2}^\infty |n_k|(C_S(s)
                M)^k\leq M.
	\end{equation}
	Then, there exists a strong solution $u \in H^2(\Omega)$ to
        \cref{npb} with $\lVert u \rVert_{H^s} \leq M$.
\end{theorem}
\begin{proof}
	Let $\lVert u \rVert_{H^s} \leq M$. Then, a similar argument
        to \cref{brouwer0} yields $\lVert A(u)\rVert_{H^s} \leq M$. By
        Schauder's fixed point theorem (\cref{fixedpointthm}) on
        $\overline{B_{H^s}(0,M)}$, a closed convex subset of
        $H^s(\Omega)$ on which $A$ is continuous and compact by
        \cref{nonlinearcompact}, we obtain $u \in
        \overline{B_{H^s}(0,M)}$ such that $A(u)=u$. Since $A$ is
        smoothing, $u \in H^2(\Omega)$.
\end{proof}

Since the existence result above is a consequence of the Schauder
fixed point theorem, no uniqueness is guaranteed. However, a more
restrictive assumption on the parameters yields uniqueness.

\begin{theorem}\label{Banach}
	Let $s\in (\frac{3}{2},2)$ and $M>0$ satisfy
        \cref{Schauder2}. Further assume
	\begin{equation}\label{Banach2}
		C_H C_S(s) \lVert \kappa^2
                \rVert_{L^\infty}|\Omega|^{\frac{1}{2}}
                \sum_{k=2}^\infty k |n_k| (C_S(s) M)^{k-1}<1.
	\end{equation}
	Then, there exists a strong solution $u \in H^2(\Omega)$ to
        \cref{npb} that is unique in $\overline{B_{H^s}(0,M)}\subseteq
        H^s(\Omega)$.
\end{theorem} 
\begin{proof}
	We invoke the Banach fixed point theorem by showing that $A$
        defines a strict contraction on
        $\overline{B_{H^s}(0,M)}$. Following the steps leading to
        \cref{continuityest}, we obtain
	\begin{equation}\label{Banach3}
		\lVert A(u)-A(v)\rVert_{H^s} \leq \Big(C_H C_S(s)
                \lVert \kappa^2
                \rVert_{L^\infty}|\Omega|^{\frac{1}{2}}
                \sum_{k=2}^\infty k |n_k| (C_S M)^{k-1}\Big)\lVert u-v
                \rVert_{H^s},
	\end{equation}
	and recall that the infinite series is $\cosh (C_S M)-1$ in \cref{Banach3}.
\end{proof}
\begin{remark}
    When $N(u) = \sinh u$, there exists $M>0$ that satisfies
    \cref{Schauder2} if and only if
\begin{equation}\label{small}
    C_H \lVert f \rVert_{L^2} + (C_HC_D+1)\lVert w \rVert_{H^2} + C_H
    \lVert \kappa^2 \rVert_{L^\infty} |\Omega|^{\frac{1}{2}}(\sinh C_S
    M_0 - C_S M_0)\leq M_0
\end{equation}
where $M_0 = C_S^{-1} \cosh^{-1}(1+\frac{1}{C_H \lVert \kappa^2
  \rVert_{L^\infty} |\Omega|^{1/2}C_S})$, which is independent of $\|
f\|_{L^2}, \|w \|_{H^2}$. Moreover, the condition in \cref{Banach2} is
equivalent to $M<M_0$. The case when the equality of \cref{small}
holds is illustrated in \cref{figure}. The equality occurs precisely
when the LHS of \cref{small}, as a function of $M$, is tangent to the
identity.
\end{remark}
\begin{remark}
    As can be shown in \cref{Schauder2,Banach2}, our fixed point
    approach works for small data where the given parameters must be
    small measured in various norms. For complexified nPBE, however,
    this restriction is necessary if we wish to preserve uniqueness of
    solution; see \cref{nonunique}. On the other hand, our approach
    establishes existence and uniqueness for a wide class of
    nonlinearities that are complex-analytic in $u$. For instance,
    nonlinearities of the form $e^u,\cosh{u}$ fall under this
    category.
\end{remark}

\section{Discussion.}\label{conclusion}

In this paper, we have studied the existence and uniqueness theory of
the complexified nPBE equation. The biggest difference between our
model and the real-valued nPBE equation stems from the non-convexity
of the nonlinearity on $\mathbb{C}$, which makes it difficult to
directly apply variational calculus to our model. In fact, for the
real-valued nPBE, we have

\begin{proposition}\cite[Theorem 2.14]{Holst1994}\label{holst}
	Let $\Omega \subseteq \mathbb{R}^3$ be open and bounded with a
        Lipschitz boundary. Let $\epsilon = diag(\tilde{\epsilon}(x))$
        where $\tilde{\epsilon} \in L^\infty(\Omega,\mathbb{R})$ with
        $\epsilon_1 \leq \tilde{\epsilon}(x) \leq \epsilon_2$ with
        $0<\epsilon_1 \leq \epsilon_2$ for all $x\in \Omega$, and let
        $\kappa \in L^\infty(\Omega,\mathbb{R})$. For every $f\in
        L^2(\Omega,\mathbb{R})$ and real-valued $g \in
        H^{\frac{1}{2}}(\partial\Omega)$, there exists a unique weak
        solution to \cref{npb} in $H^1(\Omega)$.
\end{proposition}

A good control of nPBE with complex-valued coefficients has
consequences in uncertainty quantification. In particular, if it is
shown that the solution to the nPBE is analytic in a well-defined
region with respect to a collection of stochastic parameters, then the
regularity of the solution can be precisely determined. This is
important for computing the statistics of a linear bounded Quantity of
Interest of the solution $u$ with respcet to high dimensional
stochastic parameters
\cite{nobile2008a,Castrillon2021,Castrillon2016}.  If the sequence of
approximate solutions $\{u_n\}$, given by \cref{approxsol}, is also
complex-analytic with respect to the stochastic parameters, then the
solution $u$ will also be complex-analytic in the same region. Since
$u_n$ is the solution of a linear elliptic PDE, there already exist
detailed studies of the analytic properties of $u_n$ with respect to
stochastic diffusion coefficients and random domains
\cite{babusk_nobile_temp_10,Castrillon2016,nobile2008a}.

We remark that our work has a room for improvements. In the
Debye-H\"uckel model, a collection of macromolecules such as proteins
is located in the region $\Omega_1\subseteq\mathbb{R}^3$, surrounded
by the ion-exclusion layer $\Omega_2$, which in turn is surrounded by
the solvent of positive and negative charges in
$\Omega_3$. Altogether, let $\Omega := \cup_{i=1}^3 \Omega_i$. In an
equilibrium, a well-defined potential function of the system gives
rise to a well-defined dielectric constant $\epsilon(x)$. Conversely,
we wish to study the properties of solution given a dielectric
constant. According to the Debye-H\"uckel model,
\begin{equation}\label{dielectric}
	\epsilon(x) = 
	\begin{cases}
		\epsilon_1>0, &x \in \Omega_1\\
		\epsilon_2>0, &x \in \Omega_2 \cup \Omega_3.
	\end{cases}
\end{equation}
We remark that our analysis assumes that $\epsilon$ is
Lipschitz-continuous, and therefore does not cover the case
\cref{dielectric}. The Lipschitz-continuity assumption plays a crucial
role in obtaining the elliptic regularity results such as \cref{C_H0}
and \cref{C_H00}. For now, we leave this interesting question open.

\begin{appendices}
\crefalias{section}{appsec}

\section{Estimates for the Constants.}\label{constantest}

In \cref{Schauder,Banach}, the existence and uniqueness of solutions
of \cref{npb} depend in part on the values of the constants
\(C_S(s)\), \(C_H\), and \(C_D\) described in
\cref{important_constants}. Thus having explicit estimates for these
constants is important in determining the parameter values for which
there are solutions. In this section, we demonstrate bounds for these
constants. \Cref{h1,h2,h3} are still assumed to hold throughout
\cref{constantest}.

\subsection{Estimates for $C_D$}

	\begin{lemma}
	Let $L$ be as in \cref{set-up}. Then
	\begin{equation}\label{bounded}
		 C_D \leq 2d^2 \lVert \epsilon \rVert_{W^{1,\infty}} +
                 \lVert \kappa \rVert_{L^\infty}^2.
	\end{equation}
\end{lemma}
\begin{proof}
	Since
	\[
		\lVert \kappa^2 u \rVert_{L^2} \leq \lVert \kappa^2
                \rVert_{L^\infty} \lVert u \rVert_{L^2} \leq \lVert
                \kappa^2 \rVert_{L^\infty} \lVert u \rVert_{H^2},
	\]
	it suffices to estimate $\lVert \nabla \cdot (\epsilon \nabla u)\rVert_{L^2}$. By the triangle inequality,
	\[
	\begin{split}
		\lVert \nabla \cdot (\epsilon \nabla u)\rVert_{L^2} &=
                \lVert \sum_{i,j} \partial_i (\epsilon^{ij}\partial_j
                u) \rVert_{L^2} \leq \sum_{i,j} \lVert \partial_i
                (\epsilon^{ij}\partial_j u) \rVert_{L^2} \leq
                \sum_{i,j} \lVert \partial_i \epsilon^{ij} \partial_j
                u \rVert_{L^2} + \lVert \epsilon^{ij} \partial_{ij}u
                \rVert_{L^2}\\ &\leq \sum_{i,j} (\lVert \partial_i
                \epsilon^{ij} \rVert_{L^\infty} + \lVert \epsilon^{ij}
                \rVert_{L^\infty}) \lVert u \rVert_{H^2} \leq 2d^2
                \lVert \epsilon \rVert_{W^{1,\infty}} \lVert u
                \rVert_{H^2},
		\end{split}
	\]
	and hence \cref{bounded}.
\end{proof}

\subsection{Estimates for $C_S(2)$.}

In this subsection, we are interested in obtaining an upper bound of
the operator norm of $H^2(\Omega)\hookrightarrow L^\infty(\Omega)$
where $\Omega \subseteq \mathbb{R}^3$. To obtain this Sobolev
inequality constant, a standard trick is to obtain the desired
constant for the full domain $\mathbb{R}^d$. Any reasonably regular
function defined on $\Omega$ can be extended to $\mathbb{R}^d$ via an
extension operator. Composing these two, one obtains a Sobolev
inequality on $\Omega$. See \cite[Chapter 5]{evans2010partial} for an
exposition of this material. To apply the estimates obtained in
\cite{mizuguchi2017estimation}, we lay out the following notation. For
$1 \leq p < q \leq \infty$, let $C_{p,q},D_{p,q} >0$ such that for
every $u \in W^{1,p}(\Omega)$ and $u_\Omega \coloneqq
|\Omega|^{-1}\int_\Omega u$,
\begin{equation}\label{mizuguchi}
	\lVert u \rVert_{L^q(\Omega)} \leq C_{p,q} \lVert u
        \rVert_{W^{1,p}(\Omega)},\:\lVert u - u_\Omega
        \rVert_{L^q(\Omega)} \leq D_{p,q} \lVert \nabla u
        \rVert_{L^p(\Omega)}.
\end{equation}

To estimate $C_{2,p}$ and $C_{p,\infty}$, we cite

\begin{lemma}\cite[Theorem 2.1]{mizuguchi2017estimation}\label{mizuguchi3}
  For $\Omega \subseteq \mathbb{R}^d$, if $D_{p,q}>0$ is given as
  \cref{mizuguchi}, then
	\[
		C_{p,q} = 2^{1-\frac{1}{p}}
                \max(|\Omega|^{\frac{1}{q}-\frac{1}{p}},D_{p,q}).
	\] 
\end{lemma}

The estimation for the Sobolev embedding constant, therefore, reduces
to computing $D_{p,q}$, which is summarized in the following two
lemmas:

\begin{lemma}\cite[Theorem 3.2]{mizuguchi2017estimation}\label{mizuguchi4}
  Let $p \in (2,6]$ and $u \in H^1(\Omega)$ where we further suppose
        that $\Omega$ is convex. Then, we have $\lVert u-u_\Omega
        \rVert_{L^p(\Omega)} \leq D_{2,p} \lVert \nabla u
        \rVert_{L^2(\Omega)}$ with
	\begin{equation}\label{sobconst1}
		D_{2,p}
                =\frac{d_\Omega^{1+\frac{3(p+2)}{2p}}\pi^{\frac{3(p+2)}{4p}}}{3|\Omega|}
                \frac{\Gamma(\frac{3(p-2)}{4p})}{\Gamma(\frac{3(p+2)}{4p})}\sqrt{\frac{\Gamma(\frac{3}{p})}{\Gamma(\frac{3(p-1)}{p})}}\bigg(\frac{4}{\sqrt{\pi}}\bigg)^{\frac{p-2}{2p}}.
	\end{equation}
	Hence
	\[
		C_{2,p} =
                2^{\frac{1}{2}}\max(|\Omega|^{\frac{1}{p}-\frac{1}{2}},D_{2,p}).
	\]
\end{lemma}

\begin{lemma}\cite[Theorem 3.4]{mizuguchi2017estimation}
	\label{mizuguchi5}
	For $p >3$ and $u \in W^{1,p}(\Omega)$, we have $\lVert
        u-u_\Omega \rVert_{L^\infty} \leq D_{p,\infty} \lVert \nabla u
        \rVert_{L^p}$ with
	\begin{equation}\label{sobconst3}
		D_{p,\infty} = \frac{d_\Omega^3}{3|\Omega|}
                \left|\left| |x|^{-2}\right|\right|_{L^{p^\prime}(V)},
	\end{equation}
	where $\Omega_x \coloneqq \left\{ x-y: y \in \Omega\right\}$ and $V = \bigcup\limits_{x \in \Omega} \Omega_x$.\footnote{$p^\prime \coloneqq \frac{p}{p-1}$ denotes the H\"older conjugate of $p$.} Hence
	\[
		C_{p,\infty} =
                2^{1-\frac{1}{p}}\max(|\Omega|^{-\frac{1}{p}},D_{p,\infty}).
	\]
\end{lemma}

Our proof for the existence of solution depends on the size of the
Sobolev inequality constant.

\begin{lemma}\label{sobolevestimate}
	For every $p \in (3,6)$ and $\Omega \subseteq \mathbb{R}^3$
        bounded and convex, we have
	\[
		|\Omega|^{-\frac{1}{2}}\leq C_S(2) \leq
                2^{\frac{1}{p}} C_{2,p}C_{p,\infty},
	\]
	where for $1 \leq p < q \leq \infty$, denote $C_{p,q}>0$ by a
        constant such that for every $u \in W^{1,p}(\Omega)$
	\[
		\lVert u \rVert_{L^q(\Omega)} \leq C_{p,q} \lVert u
                \rVert_{W^{1,p}(\Omega)}.
	\]
	If we further assume that $|\Omega| = Cd_\Omega^3$ for some
        $C>0$, then there exists $d_0>0$ such that for every $d_\Omega
        \leq d_0$,
	\[
		|\Omega|^{-\frac{1}{2}} \leq C_S(2) \leq
                2^{\frac{3}{2}}|\Omega|^{-\frac{1}{2}}.
	\]
\end{lemma}

\begin{proof}
	Consider the embedding $H^2(\Omega)\hookrightarrow
        W^{1,p}(\Omega)\hookrightarrow L^\infty(\Omega)$, which is
        continuous when $p\in (3,6)$. From the first embedding, we
        obtain
	\[
	\begin{split}
		\lVert u \rVert_{W^{1,p}}^p &= \lVert u \rVert_{L^p}^p + \sum_{i=1}^d \lVert \partial_i u \rVert^p\\
		&\leq C_{2,p}^p \lVert u \rVert_{H^1}^p + C_{2,p}^p \sum_{i=1}^d \lVert \partial_i u \rVert_{H^1}^p \leq C_{2,p}^p \lVert u \rVert_{H^1}^p + C_{2,p}^p \Big(\sum_{i=1}^d \lVert \partial_i u \rVert_{H^1}^2\Big)^{p/2}\\
		&\leq C_{2,p}^p \lVert u \rVert_{H^1}^p + C_{2,p}^p \lVert u \rVert_{H^2}^p \leq 2 C_{2,p}^p \lVert u \rVert_{H^2}^p,
		\end{split}
	\]
	and from the second embedding,
	\[
		\lVert u \rVert_{L^\infty} \leq C_{p,\infty} \lVert u
                \rVert_{W^{1,p}}.
	\]
	Combining the two, we obtain the upper bound. For the lower
        bound, consider a family of constant functions defined on
        $\Omega$. Then
	\[
		\frac{\lVert c \rVert_{L^\infty}}{\lVert c
                  \rVert_{H^2}} = \frac{\lVert c
                  \rVert_{L^\infty}}{\lVert c \rVert_{L^2}} =
                |\Omega|^{-\frac{1}{2}} \leq C_S(2).
	\]
	Now we assume $|\Omega| = C d_\Omega^3$ for some $C>0$ and
        give a sharp bound for $C_S(2)$ for $d_\Omega$ sufficiently
        small. Note that this hypothesis includes domains such as a
        ball $B(0,R)\subseteq \mathbb{R}^3$ or a cube $[-R,R]^3$ for
        $R>0$.
	
	Since $D_{2,p} = C(p)d_\Omega^{\frac{3}{p}-\frac{1}{2}}$ by
        \cref{sobconst1} and $|\Omega|^{\frac{1}{p}-\frac{1}{2}} = (C
        d_\Omega^3)^{\frac{1}{p}-\frac{1}{2}}$, we have
	\begin{equation}\label{sobconst5}
		C_{2,p} = 2^{\frac{1}{2}} |\Omega|^{\frac{1}{p}-\frac{1}{2}},
	\end{equation}
	for all $d_\Omega \leq d_0(p)$ for some $d_0(p)>0$. On the
        other hand, we may translate the domain and assume
        $\frac{d_\Omega}{2} = \sup\limits_{x\in \Omega} |x|$. Then, $V
        \subseteq B(0,d_\Omega)$ and
	\[
		\left|\left| |x|^{-2}
                \right|\right|_{L^{p^\prime}(V)}^{p^\prime} \leq
                \left|\left| |x|^{-2}
                \right|\right|_{L^{p^\prime}(B(0,d_\Omega))}^{p^\prime}
                = 4\pi \int_0^{d_\Omega} r^{-2p^\prime+2} dr =
                \frac{4\pi}{-2p^\prime + 3} d_\Omega^{-2p^\prime + 3},
	\]
	and thus $D_{p,\infty} = C^\prime(p)d_\Omega^{-\frac{3}{p}+1}$
        by \cref{sobconst3}, and we have
	\begin{equation}\label{sobconst6}
		C_{p,\infty} =
                2^{1-\frac{1}{p}}|\Omega|^{-\frac{1}{p}},
	\end{equation}
	for all $d_\Omega \leq d_0^\prime(p)$ for some
        $d_0^\prime(p)>0$. Combining \cref{sobconst5} and
        \cref{sobconst6}, we have
	\[
		|\Omega|^{-\frac{1}{2}} \leq C_S(2) \leq
                2^{\frac{3}{2}}|\Omega|^{-\frac{1}{2}},
	\]
	for all $d_\Omega$ sufficiently small.
\end{proof}

\subsection{Estimates for $C_H$.}\label{Ch}

To do a numerical simulation, it is of interest to obtain an estimate
for the elliptic regularity constant $C_H>0$. In applications, the
tensor $\epsilon$ is usually assumed to be a scalar-valued function,
in which case, an estimate for $C_H$ can be obtained by the Fourier
transform:
\[ \hat{f}(\xi) = \int_{\mathbb{R}^d} f(x)e^{-ix\cdot \xi}
d\xi \quad \text{and} \quad f(x) = (2\pi)^{-d}\int_{\mathbb{R}^d}
\hat{f}(\xi)e^{i x \cdot \xi} d\xi.
\]
For any $s\in \mathbb R$, define
\[
	H^s(\mathbb{R}^d) = \{f \in \mathscr{S}^\prime: \langle \xi
        \rangle^s \hat{f} \in L^2(\mathbb{R}^d)\},
\]
where $\langle \xi \rangle \coloneqq (1+|\xi|^2)^{\frac{1}{2}}$ and
$\mathscr{S}^\prime$ is the space of tempered distributions.

\begin{lemma}\label{C_H0}
	Let \(L\) be as in \cref{set-up}. Furthermore, suppose
        $\epsilon^{ij} = \epsilon(x)\delta_{ij}$ where $\delta_{ij}$
        is the Kronecker delta function and $\epsilon \in
        W^{1,\infty}(\Omega)$ such that $\Real(\epsilon(x)) \geq
        \theta>0$ for all $x \in \Omega$. Then
\begin{equation}\label{ch.estimate}
C_H \leq \frac{\lambda_1^{-1} \langle \lambda_1^{\frac{1}{3}}
  \rangle^3 }{\theta}\left(1+ \frac{\lVert \kappa^2
  \rVert_{L^\infty(\Omega)}+d^{\frac{1}{2}}\max\limits_{1 \leq i \leq
    d} \lVert \partial_i \epsilon
  \rVert_{L^\infty(\Omega)}\lambda_1^{\frac{1}{2}}}{\theta \lambda_1 -
  \mu}\right).
\end{equation}
\end{lemma}
\begin{proof}
	Given $F \in L^2(\Omega)$ and a unique weak solution $u \in
        H^1_0(\Omega)$ of the Laplace equation $-\Delta u = F$ in
        $\Omega$, we find \(C_1>0\) such that $\lVert u
        \rVert_{H^2(\Omega)} \leq C_1 \lVert F
        \rVert_{L^2(\Omega)}$. We use this energy estimate to handle
        the more complicated case.
	
	By the density argument, it suffices to assume $F \in
        C^\infty_c(\Omega)$. By an integration-by-parts argument, it
        can be shown that $u \in C^2_c(\Omega)$. Hence, we extend $u$
        to a function in $C^2_c(\mathbb{R}^d)$, which we continue to
        call $u$, by defining $u(x) = 0$ for $x \in \mathbb{R}^d
        \setminus \mathrm{supp}(u).$ Then,
	\begin{equation*}
	\begin{aligned}
		\lVert u \rVert_{H^2(\Omega)}^2 &\leq \lVert u
                \rVert_{H^2(\mathbb{R}^d)}^2= \int_{\mathbb{R}^d}
                \langle \xi \rangle^4 |\hat{u}(\xi)|^2 d\xi \\&=
                \int_{|\xi| \geq c} \langle \xi \rangle^4
                |\hat{u}(\xi)|^2 d\xi + \int_{|\xi|<c} \langle \xi
                \rangle^4 |\hat{u}(\xi)|^2 d\xi =: I + II
	\end{aligned}
	\end{equation*}
	for some \(c>0\) to be fixed later. For the high frequencies,
    \[
	\begin{split}
		I &= \int_{|\xi| \geq c} \langle \xi \rangle^4
                |\hat{u}(\xi)|^2 d\xi = \int_{|\xi|\geq c}
                \frac{\langle \xi \rangle^4}{|\xi|^4} |\widehat{\Delta
                  u}|^2 d\xi= \int_{|\xi|\geq c} \frac{\langle \xi
                  \rangle^4}{|\xi|^4} |\hat{F}(\xi)|^2 d\xi \leq
                \frac{\langle c \rangle^4}{|c|^4} \lVert F
                \rVert_{L^2(\Omega)}^2.
	\end{split}
	\]
	Combining the Poincar\'e inequality and the weak form of the
        Laplace equation, we have
	\[
		\lVert u \rVert_{L^2(\Omega)}^2 \leq \lambda_1^{-1}
                \lVert \nabla u \rVert_{L^2(\Omega)}^2 \leq
                \lambda_1^{-1}\lVert F \rVert_{L^2(\Omega)} \lVert u
                \rVert_{L^2(\Omega)},
	\]
	and therefore, for the low frequencies,
	\[
		II \leq \langle c \rangle^4 \lVert u
                \rVert_{L^2(\Omega)}^2 \leq \langle c \rangle^4
                \lambda_1^{-2} \lVert F \rVert_{L^2(\Omega)}^2.
	\]
	Combining $I$ and $II$,
	\[
		\lVert u \rVert_{H^2(\Omega)} \leq \langle c \rangle^2
                (|c|^{-4}+ \lambda_1^{-2})^{\frac{1}{2}} \lVert F
                \rVert_{L^2(\Omega)}.
	\]
	Noting that $c \mapsto \langle c \rangle^2 (|c|^{-4}+
        \lambda_1^{-2})^{\frac{1}{2}}$ has a global minimum at $c =
        \lambda_1^{\frac{1}{3}}$, we fix that value of $c$ to obtain
	\begin{equation}\label{C_H2}
		\lVert u \rVert_{H^2(\Omega)} \leq C_1 \lVert F
                \rVert_{L^2(\Omega)} \quad \text{where} \quad C_1
                \coloneqq \lambda_1^{-1} \langle
                \lambda_1^{\frac{1}{3}} \rangle^3 .
	\end{equation}
	Now we assume $u \in H^1_0(\Omega)$ is the unique weak
        solution of
	\begin{equation}\label{C_H}
		-\nabla \cdot (\epsilon\nabla u) + \kappa^2 u = f
                \quad \text{in } \Omega.
	\end{equation}
	Setting $F \coloneqq f- \kappa^2 u \in L^2(\Omega)$, the
        product rule applied to \cref{C_H} yields
	\[
		-\Delta u = \epsilon(x)^{-1}(F + \nabla \epsilon \cdot
                \nabla u).
	\]
	Noting that $|\epsilon(x)| \geq |\Real (\epsilon(x))| \geq
        \Real (\epsilon(x)) \geq \theta$, an immediate application of
        \cref{C_H2} yields
	\begin{equation}\label{C_H3}
		\lVert u \rVert_{H^2(\Omega)} \leq \frac{C_1}{\theta}
                (\lVert F \rVert_{L^2(\Omega)} + \lVert \nabla
                \epsilon \cdot \nabla u\rVert_{L^2(\Omega)}).
	\end{equation}
	Taking the real part of the weak form of \cref{C_H}, we have
	\[
		\int_{\Omega} \Real(\epsilon(x)) |\nabla u|^2 +
                \int_{\Omega} \Real(\kappa^2) |u|^2 = \Real
                \int_{\Omega} f\overline{u}.
	\]
	Recalling that $\Real(\kappa^2(x)) \geq - \mu$ for all $x \in
        \Omega$ and the uniform ellipticity,
	\begin{equation}\label{poincare3}
		\theta \int_{\Omega} |\nabla u|^2 + \int_{\Omega}
                \Real(\kappa^2) |u|^2 \leq \lVert f
                \rVert_{L^2(\Omega)} \lVert u \rVert_{L^2(\Omega)}.
	\end{equation}
	The Poincar\'e inequality yields $ (\theta \lambda_1 -
        \mu)\lVert u \rVert_{L^2(\Omega)}^2 \leq \theta \int_{\Omega}
        |\nabla u|^2 + \int_{\Omega} \Real(\kappa^2) |u|^2$, which
        gives
	\[
		\lVert u \rVert_{L^2(\Omega)} \leq \frac{\lVert f
                  \rVert_{L^2(\Omega)}}{\theta \lambda_1 -\mu}.
	\]
	Another application of the Poincar\'e inequality to
        \cref{poincare3} yields $(\theta - \mu \lambda_1^{-1})
        \int_\Omega |\nabla u|^2 \leq \theta \int_{\Omega} |\nabla
        u|^2 + \int_{\Omega} \Real(\kappa^2) |u|^2$, which gives
	\begin{equation}\label{l22}
		\lVert \nabla u \rVert_{L^2(\Omega)} \leq
                \lambda_1^{\frac{1}{2}}\frac{\lVert f
                  \rVert_{L^2(\Omega)}}{\theta \lambda_1 - \mu}.
	\end{equation}
	Hence,
	\begin{equation}\label{C_H4}
		\lVert F \rVert_{L^2(\Omega)} \leq \Big(1+\frac{\lVert
                  \kappa^2 \rVert_{L^\infty(\Omega)}}{\theta \lambda_1
                  - \mu}\Big) \lVert f \rVert_{L^2(\Omega)}.
	\end{equation}
	On the other hand, the Cauchy-Schwarz inequality yields
	\begin{equation}\label{C_H5}
		\lVert \nabla \epsilon \cdot \nabla
                u\rVert_{L^2(\Omega)} \leq \sqrt{d} \max_{1 \leq i
                  \leq d} \lVert \partial_i
                \epsilon\rVert_{L^\infty(\Omega)} \lVert \nabla u
                \rVert_{L^2(\Omega)} \leq \sqrt{d} \max_{1 \leq i \leq
                  d} \lVert \partial_i
                \epsilon\rVert_{L^\infty(\Omega)}\lambda_1^{\frac{1}{2}}\frac{\lVert
                  f \rVert_{L^2(\Omega)}}{\theta \lambda_1 - \mu}.
	\end{equation}
	By \cref{C_H2,C_H3,C_H4,C_H5},
	\[
		\lVert u \rVert_{H^2(\Omega)} \leq
                \frac{\lambda_1^{-1} \langle \lambda_1^{\frac{1}{3}}
                  \rangle^2 (1+
                  \lambda_1^{\frac{2}{3}})^{\frac{1}{2}}}{\theta}\Big(1+
                \frac{\lVert \kappa^2
                  \rVert_{L^\infty(\Omega)}+d^{\frac{1}{2}}\max\limits_{1
                    \leq i \leq d} \lVert \partial_i \epsilon
                  \rVert_{L^\infty(\Omega)}\lambda_1^{\frac{1}{2}}}{\theta
                  \lambda_1 - \mu}\Big) \lVert f \rVert_{L^2(\Omega)}.
	\]
\end{proof}

In general when $\epsilon$ is a tensor, a direct application of
Fourier transform seems infeasible. Instead, we closely follow the
argument of \cite[Section 6.3, Theorem 4]{evans2010partial} to obtain
an estimate on $C_H$. An emphasis here is that we keep track of the
implicit constants.

\begin{lemma}\label{C_H00}
	Let \(L\) be as given in \cref{set-up}. Then
\begin{equation}\label{ch.estimate2}
C_H \leq N(\Omega)\Big(\Big(\frac{1+\lambda_1}{\theta \lambda_1 - \mu}
\Big)^2 + d C_0\Big)^{\frac{1}{2}},
\end{equation}
where $N(\Omega) \in \mathbb{N}$ and $C_0$ is defined in \cref{C_0}.
\end{lemma}
\begin{remark}
For every $\Omega$ such that $\lambda_1 \simeq d_{\Omega}^{-2}$, the
RHS of \cref{ch.estimate,ch.estimate2} converge to
$\frac{C(\Omega)}{\theta}$ as $d_\Omega \rightarrow 0$.
\end{remark}
\begin{remark}
Since the estimate of \cref{ch.estimate2} depends on the number of
finitely many open balls covering $\Omega$, the geometry of $\partial
\Omega$ plays a big role in the computation of $N(\Omega)$. This will
be pursued in future research.
\end{remark}
\begin{proof}
	Let $V \Subset W \Subset \Omega$. Let $\zeta \in
        C^\infty_c(\Omega)$ such that $0\leq \zeta \leq 1$ and $\zeta
        = 1$ on $V$, and $\mathrm{supp}(\zeta)\subseteq W$. Set $F =
        f- \kappa^2 u \in L^2(\Omega)$. Consider the weak form of
    \begin{equation}
		-\nabla \cdot (\epsilon\nabla u) + \kappa^2 u = f
                \quad \text{in } \Omega
	\end{equation}
	applied to the test function $\phi = -D_k^{-h} \zeta^2 D_k^h
        u$, where
	\[D_k^h u(x)= \frac{u(x+he_k) - u(x)}{h}\] and $\{e_k\}_{k=1}^d$
        forms the standard basis of $\mathbb{R}^d$. Using integration
        by parts and the product rule of discrete derivatives,
	\begin{equation*}
	\begin{aligned}
		\int_\Omega \epsilon^{ij} \partial_j u
                \overline{\partial_i \phi} &= \int_\Omega D_k^h
                (\epsilon^{ij}\partial_j u) \overline{\partial_i
                  (\zeta^2 D_k^h u)} \\ &= \int_{\Omega}
                (\epsilon^{ij,h} D_k^h \partial_j u + D_k^h
                \epsilon^{ij} \partial_j u) (\overline{2\zeta
                  \partial_i \zeta D_k^h u} + \overline{\zeta^2 D_k^h
                  \partial_i u})\\ &= \int_\Omega \zeta^2
                \epsilon^{ij,h} D_k^h \partial_j u \overline{D_k^h
                  \partial_i u} + R,
	\end{aligned}
	\end{equation*}
	where $\epsilon^{ij,h}(x) \coloneqq \epsilon^{ij}(x+he_k)$. By
        uniform ellipticity,
	\[
		\Real\int_\Omega \zeta^2 \epsilon^{ij,h} D_k^h
                \partial_j u \overline{D_k^h \partial_i u} \geq \theta
                \int_\Omega \zeta^2 |D_k^h \nabla u|^2.
	\]
	The other three products are estimated above by the
        Cauchy-Schwarz inequality:
	\begin{equation}\label{R}
	\begin{split}
		R &\leq \left|\int_\Omega 2 \zeta \partial_i \zeta
                \epsilon^{ij,h} D_k^h \partial_j u \overline{D_k^h
                  u}\right| + \left| \int_\Omega 2\zeta \partial_i
                \zeta D_k^h \epsilon^{ij} \partial_j u \overline{D_k^h
                  u}\right| + \left|\int_\Omega D_k^h \epsilon^{ij}
                \zeta^2 \partial_j u \overline{D_k^h \partial_i u}
                \right| \\ &\leq 2 \lVert \nabla \zeta
                \rVert_{L^\infty(\Omega)} \lVert \epsilon
                \rVert_{W^{1,\infty}(\Omega)}\Big( \int_\Omega \zeta
                |D_k^h \nabla u||D_k^h u| + \int_\Omega \zeta |\nabla
                u||D_k^h u|\Big)+ \lVert \epsilon
                \rVert_{W^{1,\infty}(\Omega)} \int_\Omega \zeta
                |\partial_j u| |D_k^h \partial_i u|.
	\end{split}
	\end{equation}
	Recalling the following variant of Cauchy-Schwarz inequality
	\[
		ab \leq \frac{a^2}{2\delta} + \frac{\delta b^2}{2},
	\]
	for $a,b \geq 0$ and $\delta>0$ and the following control of
        discrete derivatives with respect to the continuous
        derivatives for sufficiently small $|h|>0$,
	\begin{equation}\label{discretederiv}
		\lVert D_k^h \phi\rVert_{L^2(V)} \leq \lVert
                \partial_k \phi \rVert_{L^2(\Omega)},\:\forall \phi
                \in H^1(\Omega),\: V \Subset \Omega,
	\end{equation}
	\cref{R} is bounded above by
	\[
		\leq C_1\delta \int_\Omega \zeta^2 |D_k^h \nabla u|^2
                + C_2\int_{\Omega} |\nabla u|^2
	\]
	where
	\[
	\begin{split}
		C_1 &\coloneqq \lVert \epsilon
                \rVert_{W^{1,\infty}(\Omega)}\Big(\lVert \nabla \zeta
                \rVert_{L^\infty(\Omega)} + \frac{1}{2}\Big) \quad
                \text{and} \quad C_2 \coloneqq \lVert \epsilon
                \rVert_{W^{1,\infty}(\Omega)}\Big(2\lVert \nabla \zeta
                \rVert_{L^\infty(\Omega)} +\frac{1+2\lVert \nabla
                  \zeta \rVert_{L^\infty(\Omega)}}{2\delta}\Big).
	\end{split}
	\]
	Choosing $\delta = \frac{\theta}{2C_1}$, we use the triangle
        inequality to obtain
	\begin{equation}\label{lowerbound2}
		\Real\int_\Omega (\epsilon \nabla u) \cdot
                \overline{\nabla u} \geq \frac{\theta}{2}\int_\Omega
                \zeta^2 |D_k^h \nabla u|^2 - C_2 \int_\Omega |\nabla
                u|^2.
	\end{equation}
	On the other hand, we estimate the right-hand side of the weak
        form:
	\[
		\left|\int_\Omega F \overline{\phi}\right| \leq
                \frac{1}{2\delta} \int_\Omega |F|^2 +
                \frac{\delta}{2}\int_\Omega |\phi|^2,
	\]
	where the first term is estimated above as in \cref{C_H4}.
	\begin{equation*}
	\begin{aligned}
		\int_\Omega |\phi|^2 &\leq \int_\Omega |\partial_k (\zeta^2 D_k^h u)|^2 \\
		&\leq 2 \int_\Omega |2\zeta \partial_k \zeta D_k^h u|^2 + 2 \int_\Omega \zeta^2 |D_k^h \partial_k u|^2\\
		&\leq 8 \lVert \nabla \zeta \rVert_{L^\infty(\Omega)}^2 \int_\Omega |\nabla u|^2 + 2 \int_\Omega \zeta^2 |D_k^h \nabla u|^2,
	\end{aligned}
	\end{equation*}
	where the last inequality is by \cref{discretederiv}. Let $\delta = \frac{\theta}{4}$. Then,
	
	\begin{equation}\label{lowerbound3}
		\left| \int_\Omega F \overline{\phi}\right| \leq
                \frac{2}{\theta} \Big(1+\frac{\lVert \kappa^2
                  \rVert_{L^\infty(\Omega)}}{\theta \lambda_1 - \mu}
                \Big)^2 \int_\Omega |f|^2 + \theta \lVert \nabla \zeta
                \rVert_{L^\infty(\Omega)}^2 \int_\Omega |\nabla u|^2 +
                \frac{\theta}{4} \int_\Omega \zeta^2 |D_k^h \nabla
                u|^2.
	\end{equation}
	Combining \cref{lowerbound3} and \cref{lowerbound2},
	\[
		\frac{\theta}{4} \int_V |D_k^h \nabla u|^2 \leq
                \frac{\theta}{4} \int_\Omega \zeta^2 |D_k^h \nabla
                u|^2 \leq \frac{2}{\theta}\Big(1+ \frac{\lVert
                  \kappa^2 \rVert_{L^\infty(\Omega)}}{\theta \lambda_1
                  - \mu} \Big)^2 \int_\Omega |f|^2 + (C_2 + \theta
                \lVert \nabla \zeta \rVert_{L^\infty(\Omega)}^2)
                \int_\Omega |\nabla u|^2,
	\]
	and by \cref{l22},
	\begin{equation}
	\begin{split}
	\label{C_0}
		\int_V |D_k^h \nabla u|^2 &\leq C_0 \int_\Omega |f|^2,
                \\ C_0 &=
                \frac{4}{\theta}\bigg(\frac{2}{\theta}\Big(1+
                \frac{\lVert \kappa^2
                  \rVert_{L^\infty(\Omega)}}{\theta \lambda_1 - \mu}
                \Big)^2+ \frac{\lambda_1 (C_2 + \theta \lVert \nabla
                  \zeta \rVert_{L^\infty(\Omega)}^2)}{(\theta
                  \lambda_1 - \mu)^2} \bigg).
	\end{split}
	\end{equation}
	By \cite[Section 5.8.2, Theorem 3]{evans2010partial}, this
        shows $\partial_k \nabla u \in L^2(V,\mathbb{C}^d)$ for all $1
        \leq k \leq d$ with the same bound on the $L^2$-norm. Hence,
	\begin{equation}\label{h2homogeneous}
		\sum_{1\leq i,j \leq d} \lVert \partial_{ij} u
                \rVert_{L^2(V)}^2 \leq d C_0 \int_\Omega |f|^2.
	\end{equation}
	Recall that the linear theory (\cref{LM} and \cref{LM2})
        yields
	\begin{equation}\label{LM3}
		\lVert u \rVert_{H^1(\Omega)} \leq
                \frac{1+\lambda_1}{\theta \lambda_1 - \mu} \lVert f
                \rVert_{L^2(\Omega)}.
	\end{equation}
	Combining \cref{h2homogeneous} with \cref{LM3}, we obtain
	\begin{equation}\label{h2inhomogeneous}
		\lVert u \rVert_{H^2(V)} \leq
                \Big(\Big(\frac{1+\lambda_1}{\theta \lambda_1 - \mu}
                \Big)^2 + d C_0\Big)^{\frac{1}{2}}\lVert f
                \rVert_{L^2(\Omega)}.
	\end{equation}
	Since $\Omega$ is bounded, $\{x \in \Omega: \inf_{y \in
          \partial\Omega}|x-y| \geq \delta\}$ can be covered by
        finitely many open sets for every $\delta>0$. Given any point
        $y \in \partial \Omega$, there exists a diffeomorphism that
        takes a small neighborhood of $y$ (in $\overline{\Omega}$)
        into a neighborhood in the half-plane $\mathbb{R}^d_+
        \coloneqq \mathbb{R}^{d-1}\times [0,\infty)$ where $y$ is
          identified with $0 \in \mathbb{R}^d_+$. Via this
          diffeomorphism, one can show that the $H^2$-norm of $u$ in
          the neighborhood of $y$ obeys an esmate similar to
          \cref{h2inhomogeneous}. Hence, there exists $N = N(\Omega)
          \in \mathbb{N}$ such that
	\[
		\lVert u \rVert_{H^2(\Omega)} \leq
                N(\Omega)\Big(\Big(\frac{1+\lambda_1}{\theta \lambda_1
                  - \mu} \Big)^2 + d C_0\Big)^{\frac{1}{2}} \lVert f
                \rVert_{L^2(\Omega)}.
	\]
\end{proof}

\section{Failure of Uniqueness}\label{appendixB}
Most notably, our result has a smallness assumption on data $(f,g)$
and further restrictions on the given parameters; see \cref{h1,h2,h3}
in \cref{exuniq} and \cref{Schauder,Banach}. When \cref{npb} is
complexified, it is not obvious whether or not these sufficient
conditions are in fact necessary. We construct an example of nPBE that
admits multiple solutions. By construction, this family of nPBEs fails
to satisfy the invertibility condition given in \cref{h3} and/or the
smallness assumption on $(f,g) \in L^2(\Omega)\times
H^{\frac{3}{2}}(\partial \Omega)$. In particular, this example is
consistent with the well-known uniqueness result of
\cite{kwong1989uniqueness}.

We wish to obtain a radial solution $u(x) = u(|x|) = y(r)$, where
$r=|x|\geq 0$, to \cref{npb} where $\epsilon=1$ for simplicity and
$\kappa = i\tilde{\kappa} \in i \mathbb{R}$ on domain $\Omega =
B(0,R)\subseteq \mathbb{R}^d$ for $R>0, d\geq 1$ and $f(x) = \lambda
\in \mathbb{R}, g(x) =
\sinh^{-1}(\frac{\lambda}{\tilde{\kappa}^2})$. In the polar
coordinate, our example reduces to an ODE

\begin{equation}\label{nonlinearBessel}
\begin{split}
		r y^{\prime\prime}+(d-1)y^\prime + \tilde{\kappa}^2
                r\sinh y&=r\lambda,\: r \in (0,R)\\ y(R) &=
                \sinh^{-1}\Big(\frac{\lambda}{\tilde{\kappa}^2}\Big).
\end{split}
\end{equation}
where it is clear that the constant function $r\mapsto
\sinh^{-1}\Big(\frac{\lambda}{\tilde{\kappa}^2}\Big)$ is a trivial
solution. Since \cref{nonlinearBessel} is symmetric under $r\mapsto
-r$, we may consider $\lambda \geq 0$. It is also clear that $(f,g)
\in L^2(\Omega)\times H^{\frac{3}{2}}(\partial\Omega)$ can be taken as
large as possible (in norm) by taking $\lambda$ arbitrarily
large. Furthermore, we note that \cref{h3} is violated when $R \gg 1$
depending on $\tilde{\kappa}$. To elaborate, fix
$\tilde{\kappa}>0$. If \cref{h3} holds, then $\tilde{\kappa}^2 \leq
\mu < \lambda_1 = \frac{C_B}{R^2}$. Hence if $R
>\frac{C_B}{\tilde{\kappa}}$, then \cref{h3} cannot hold.

\begin{proposition}\label{nonunique}
	Let $d\geq 1, \tilde{\kappa}>0, \lambda \geq 0$. Then, there
        exists a non-trivial solution to \cref{nonlinearBessel} with
        $R>\frac{C_B}{\tilde{\kappa}}$.
\end{proposition}
	
Reducing \cref{nonlinearBessel} into a first-order ODE by introducing
$w = y^\prime$, we obtain \begin{equation}\label{phaseportrait}
		\begin{pmatrix}
			y\\
			w
		\end{pmatrix}^\prime = F(r,y,w)\coloneqq
		\begin{pmatrix}
			w\\
			-\tilde{\kappa}^2\sinh y - (d-1)\frac{w}{r} + \lambda
		\end{pmatrix}.
	\end{equation}
	
	For $d=1$, \cref{phaseportrait} admits an autonomous
        Hamiltonian vector field where the Hamiltonian is given by
	\[
		H(y,w) = \frac{w^2}{2} + \tilde{\kappa}^2(\cosh y -1)
                - \lambda y.
	\]
	Since the level sets of $H$ are a collection of closed
        one-dimensional curves, all solutions are global and
        periodic. The inner curves have lower values of $H$ than the
        outer curves. Indeed, the global minimum of $H$ occurs at
        $P=(\sinh^{-1}\Big(\frac{\lambda}{\tilde{\kappa}^2}\Big),0)$
        where $H(P)\leq 0$ with the equality if and only if $\lambda =
        0$. Hence for each initial datum $\begin{pmatrix} c\\ 0
	\end{pmatrix}$,
        there exists a unique solution $y$ to \cref{nonlinearBessel}
        where $y(R) = 0$ for infinitely many $R>0$. We include a phase
        portrait where the solutions lie on the curves of constant
        Hamiltonian.
	
	\begin{figure}[H]
	\begin{tikzpicture}

		\node at (0,0) {\includegraphics[scale=0.3, trim =
                    0.5cm 0.15cm 0.05cm 0.05cm,
                    clip]{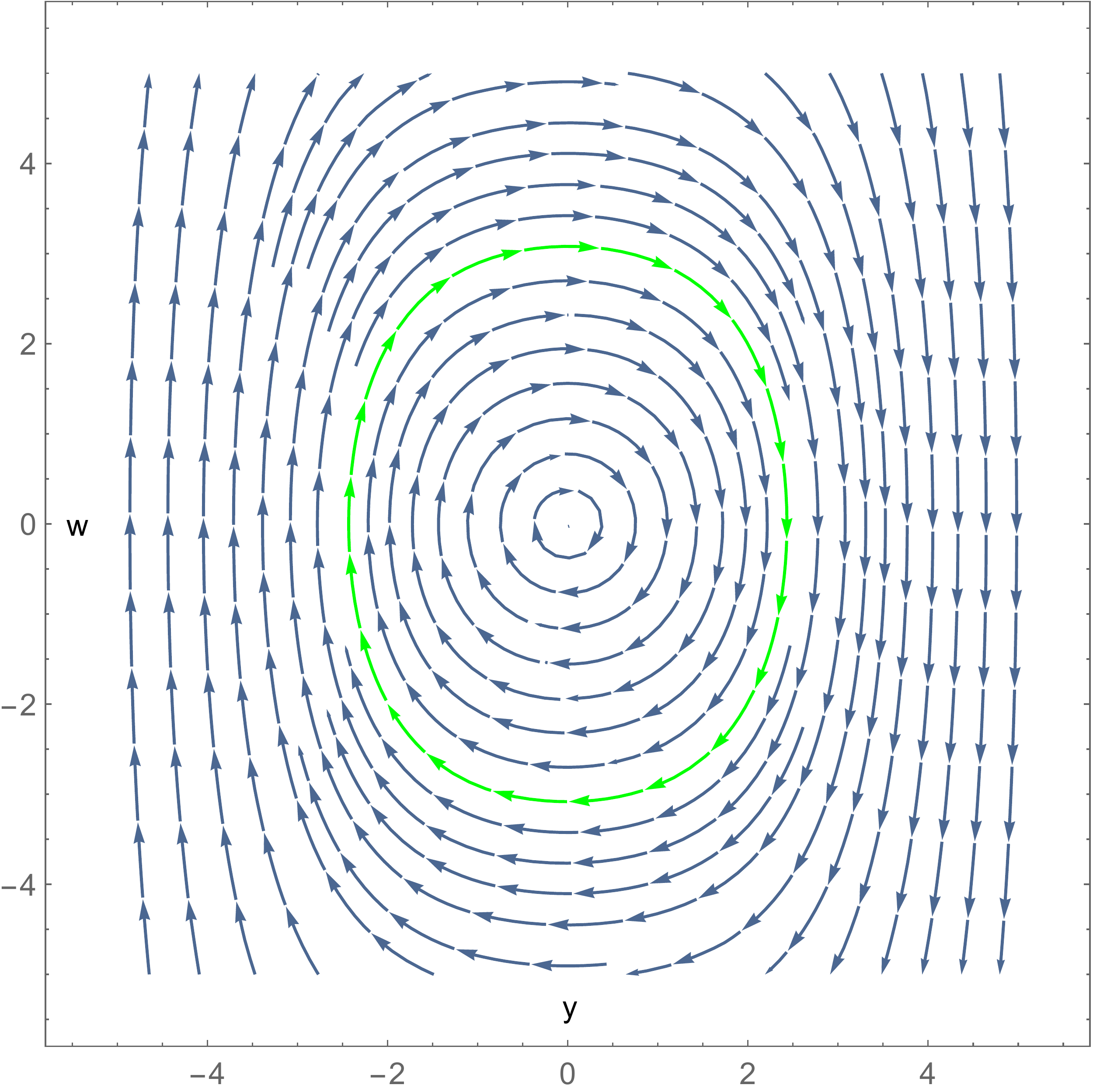}};

		\fill [white] (-0.5,-3.4) rectangle (0.5,-3.1);
		\fill [white] (-3.5,-1) rectangle (-3.1,1);
		
		\node at (-3.25,0.15) {$w$};
		\node at (0.1,-3.25) {$y$};

	    \node at (8.5,0) {\includegraphics[scale=0.3, trim = 0.5cm
                0.15cm 0.05cm 0.05cm,
                clip]{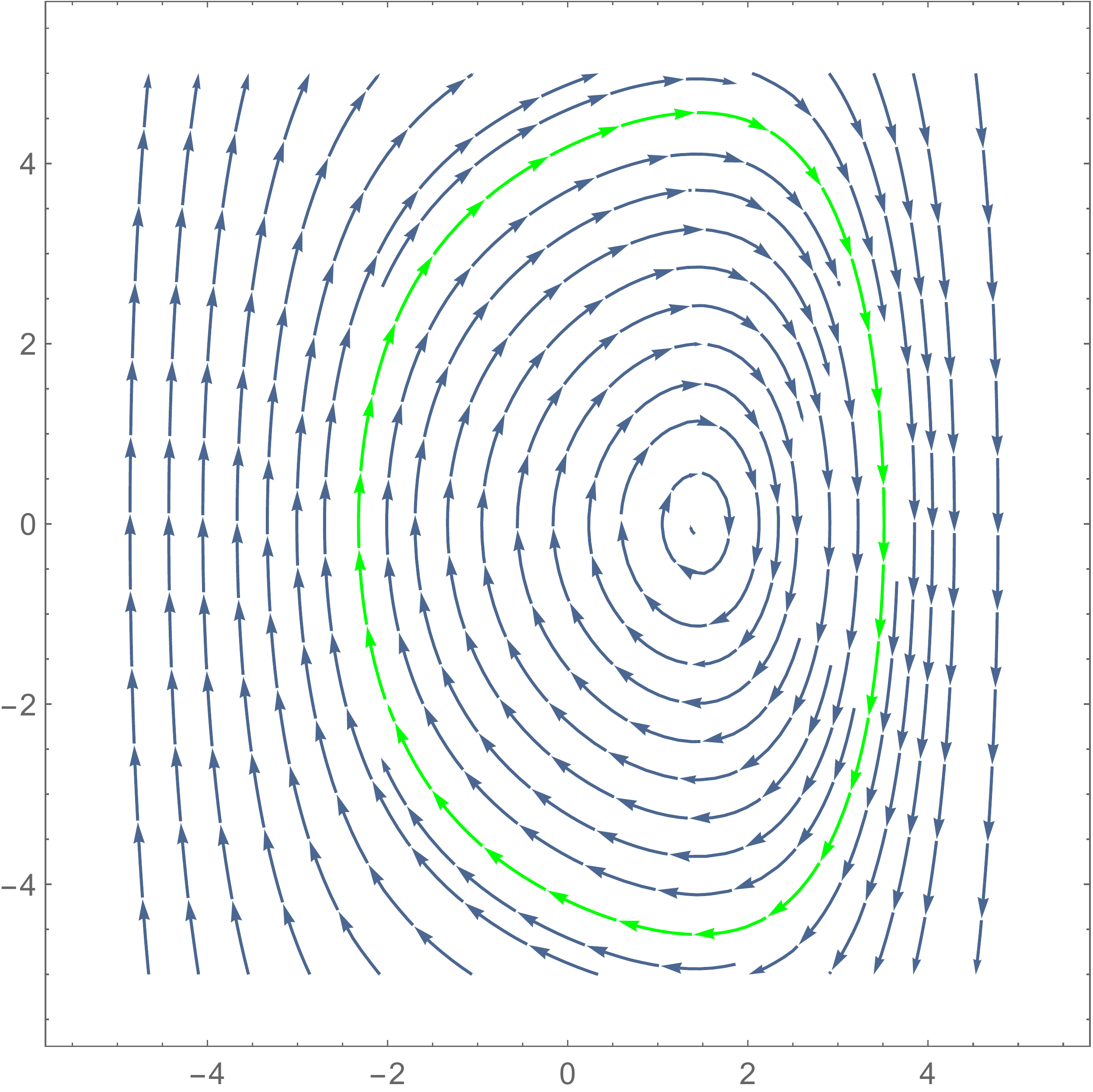}};

		\end{tikzpicture}
		\caption{Vector fields of \cref{phaseportrait} with
                  $d=1,\tilde{\kappa}=1$ with the left plot portraying
                  $\lambda=0$, and the right $\lambda=2$.}
	\end{figure}
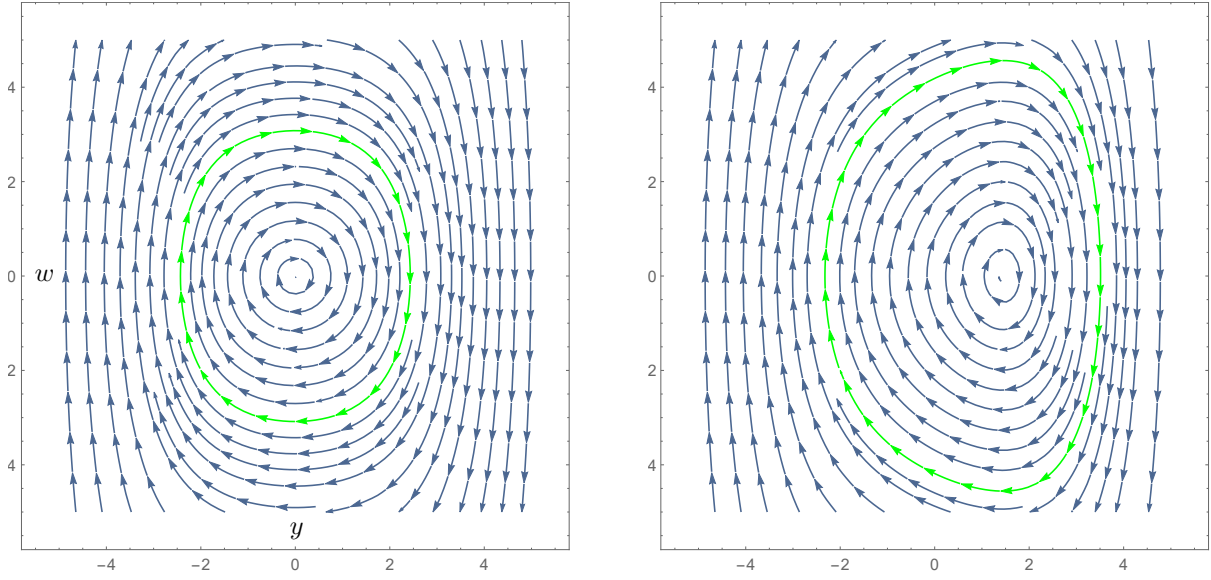

For $d \geq 2$, the vector field corresponding to \cref{phaseportrait}
is non-autonomous, and $F$ in \cref{phaseportrait} is not well-defined
at $r=0$ where our initial data are given. We regularize the ODE so
that the regularized vector field is continuous (in $r$) near $r=0$,
and show that the limiting solution satisfies
\cref{nonlinearBessel}. We solve an ODE that is slightly more general
than \cref{nonlinearBessel}. We use the notations of \cref{nonunique}.

\begin{lemma}\label{NonuniqueNonlinearBessel}
    For every $A \geq 0$ and $c\in \mathbb{R}$, there exists
    $R>\frac{C_B}{\tilde{\kappa}}$ and $y \in
    C^\infty_{\mathrm{loc}}((0,\infty);\mathbb{R})$ such that $y$
    satisfies
\begin{equation}\label{NonuniqueNonlinearBessel2}
\begin{split}
    r y^{\prime\prime}+Ay^\prime + \tilde{\kappa}^2 r\sinh
    y&=r\lambda,\: r \in (0,\infty),\\ \lim\limits_{r \to 0+} y(r) =
    c,\: \lim\limits_{r \to 0+} y^\prime(0+) = 0,\: y(R) &=
    \sinh^{-1}\Big(\frac{\lambda}{\tilde{\kappa}^2}\Big).
    \end{split}
\end{equation}
\end{lemma}

\begin{proof}[Proof of \cref{nonunique}]
    Set $A = d-1$.
\end{proof}

\begin{proof}[Proof of \cref{NonuniqueNonlinearBessel}]
    The $A=0$ case is equal to that when $d=1$, and therefore assume
    $A>0$. Moreover, assume $c \neq
    \sinh^{-1}\Big(\frac{\lambda}{\tilde{\kappa}^2}\Big)$ since it
    yields a trivial solution. For $\epsilon>0$, consider the
    perturbed ODE:
\begin{equation}\label{perturbedBessel}
\begin{split}
    (r+\epsilon)y_\epsilon^{\prime\prime} + Ay_\epsilon^\prime +
  \tilde{\kappa}^2(r+\epsilon) \sinh y_\epsilon &=
  (r+\epsilon)\lambda,\: r \in
        [-\frac{\epsilon}{2},\infty),\\ y_\epsilon(0) = c,\:
          y_\epsilon^\prime(0) &= 0,
    \end{split}
\end{equation}
which, after setting $w_\epsilon = y_\epsilon^\prime$, reduces to

\[
   \begin{pmatrix}
			y_\epsilon\\
			w_\epsilon
		\end{pmatrix}^\prime = F_\epsilon(r,y_\epsilon,w_\epsilon)\coloneqq
		\begin{pmatrix}
			w_\epsilon\\
			-\tilde{\kappa}^2\sinh y_\epsilon - \frac{A w_\epsilon}{r+\epsilon} - \lambda
		\end{pmatrix}. 
\]

Since $F_\epsilon$ is smooth in $r$ near $r=0$ and locally Lipschitz
in $(y,w)$, there exists $T_\epsilon \in (0, \frac{\epsilon}{2})$ and
$y_\epsilon \in C([-T_\epsilon,T_\epsilon];\mathbb{R})\cap
C_{\mathrm{loc}}^\infty((-T_\epsilon,T_\epsilon);\mathbb{R})$ such
that $y_\epsilon$ is a unique solution to \cref{perturbedBessel}. In
the maximal interval of existence, $\begin{pmatrix}
  y_\epsilon\\ w_\epsilon
\end{pmatrix}$
satisfies
\[
\frac{d}{dr} H(y_\epsilon(r),w_\epsilon(r)) =
w_\epsilon(r)(y_\epsilon^{\prime\prime}(r) + \tilde{\kappa}^2\sinh
y_\epsilon(r)-\lambda) = - \frac{Aw_\epsilon(r)^2}{r+\epsilon} \leq
0,\: r \geq -\frac{\epsilon}{2}
\]
and therefore the forward orbit of
$\begin{pmatrix}
y_\epsilon\\
w_\epsilon
\end{pmatrix}$
is bounded in the compact subset $\{(y,w) \in \mathbb{R}^2: H(y,w)
\leq H(c,0)\}$ on which $F_\epsilon$ is Lipschitz. Hence,
$\begin{pmatrix}
y_\epsilon\\
w_\epsilon
\end{pmatrix}$
can be uniquely extended globally in forward time, obeying the estimate
\begin{equation}\label{uniformbound}
    H(y_\epsilon(r),w_\epsilon(r)) \leq H(c,0),\: r \geq 0.
\end{equation}

This global bound on $|y_\epsilon|+|w_\epsilon|$ yields an existence
of a limit function, since for $r_1, r_2 \geq 0$,
\[
|y_\epsilon(r_2) - y_\epsilon(r_1)| = \left|\int_{r_1}^{r_2}
w_\epsilon(\rho)d\rho\right| \leq C |r_2-r_1|,
\]
where $C>0$ is independent of $\epsilon>0$. An immediate application
of Arzel\`{a}-Ascoli Theorem implies that there exists a subsequence
$\epsilon_k>0$ that tends to zero (from the right) and $y \in
C_{\mathrm{loc}}([0,\infty);\mathbb{R})$ such that
  $y_{\epsilon_k}\xrightarrow[k\to 0]{}y$ in the topology of uniform
  convergence on compact subsets; in particular, $y(0)=c$.

Let $T>0$. Since $\{w_{\epsilon_k}\}$ is uniformly bounded in
$L^2((0,T);\mathbb{R})$ due to \cref{uniformbound}, there exists a
subsequence of $\{\epsilon_k\}$ and $w \in L^2((0,T);\mathbb{R})$ such
that, possibly after relabelling the subsequence, $w_{\epsilon_k}
\rightharpoonup w$ in $L^2((0,T);\mathbb{R})$. This weak convergence
of derivatives and the uniform convergence $y_{\epsilon_k} \rightarrow
y$ on $[0,T]$ implies that $w$ is the weak derivative of
$y$. Furthermore, we have
$((r+\epsilon_k)y_{\epsilon_k}^\prime)^\prime (r) = (1-A)
y_{\epsilon_k}^\prime - \tilde{\kappa}^2(r+\epsilon_k) \sinh
y_{\epsilon_k}(r) + (r+\epsilon_k)\lambda$ from \cref{perturbedBessel}
where the right-hand side is uniformly bounded in
$L^2((0,T);\mathbb{R})$. Another application of the Arzel\`{a}-Ascoli
Theorem implies that there exists $Y \in C([0,T];\mathbb{R})$ such
that $(\cdot+\epsilon_k)y_{\epsilon_k}^\prime \xrightarrow[k\to
  \infty]{} Y$ in $C([0,T];\mathbb{R})$, possibly after relabelling
the subsequence, and follows $y_{\epsilon_k}^\prime \xrightarrow[k \to
  \infty]{} \frac{Y}{r}$ in $C([\delta,T];\mathbb{R})$ for every
$\delta>0$, and therefore we identify $w(r)$ with a continuous
function $\frac{Y(r)}{r}$ on $(0,T)$; indeed, $w = y^\prime$
classically on $(0,T)$. Yet another application of \cref{uniformbound}
and the triangle inequality $|w(r)| \leq
|w(r)-y_{\epsilon_k}^\prime(r)| + |y_{\epsilon_k}^\prime(r)|$ yields
the bound $|w(r)| \leq M$ for some $M>0$ on $(0,T)$.

Since $y_{\epsilon_k}$ is a classical solution to
\cref{perturbedBessel}, it is also a weak solution. Writing
\cref{perturbedBessel} in the weak form, integrating by parts, and
taking $k \rightarrow \infty$, we obtain $ry^{\prime\prime} +
Ay^\prime +\tilde{\kappa}^2 r\sinh y = r\lambda$ on $(0,T)$ in the
weak sense where the distributional derivative $y^{\prime\prime}$ can
be identified with a continuous function on $(0,T)$ using the equation
above. Using \cref{perturbedBessel} and the uniform convergence of
$y_{\epsilon_k}$ and its derivative as $k \rightarrow \infty$, we
conclude $(r+\epsilon_k) y_{\epsilon_k}^{\prime\prime}
\xrightarrow[k\to \infty]{} -(Ay^\prime + \tilde{\kappa}^2r \sinh y) +
r\lambda = ry^{\prime\prime}$ uniformly on $[\delta,T]$, and therefore
$y_{\epsilon_k}^{\prime\prime} \xrightarrow[k\to
  \infty]{}y^{\prime\prime}$ in $C([\delta,T];\mathbb{R})$ for every
$\delta>0$. We have shown that $y_{\epsilon_k}^{(j)} \xrightarrow[k
  \to \infty]{} y^{(j)}$ uniformly on compact subsets of $(0,T)$ for
$j=0,1,2$. Taking $k \to \infty$ from \cref{perturbedBessel}, we
conclude that $y$ satisfies the desired ODE pointwise on
$(0,T)$. Since the vector field $F$ is smooth on $(0,T) \times
\mathbb{R}^2$ where $T>0$ was arbitrary, we conclude $y \in
C^\infty_{\mathrm{loc}}((0,\infty);\mathbb{R}))$.

Since $y^{\prime\prime}$ is continuous on $(0,T)$ and
$y_{\epsilon_k}^\prime \rightharpoonup y^\prime$ in
$L^2((0,T);\mathbb{R})$ as $k \to \infty$, for every $\phi \in
C_c^\infty((0,T);\mathbb{R})$,
\[
    \int_0^T y_{\epsilon_k}^{\prime\prime}\phi = - \int_0^T
    y_{\epsilon_k}^\prime \phi^\prime \xrightarrow[k\to \infty]{} -
    \int_0^T y^\prime \phi^\prime = \int_0^T y^{\prime\prime}\phi.
\]
Therefore, $\{y_{\epsilon_k}^{\prime\prime}\}$ is uniformly bounded in
$L^2((0,T);\mathbb{R})$, and another application of the
Arzel\`{a}-Ascoli Theorem shows that there exists a convergent
subsequence of $\{y_{\epsilon_k}^\prime\}$ in
$C([0,T];\mathbb{R})$. Since we showed $y_{\epsilon_k}^\prime
\xrightarrow[k\to \infty]{C([\delta,T];\mathbb{R})} y^\prime$ for
every $\delta>0$, we conclude that $\delta$ could be taken to be
zero. In particular, $y^\prime(0) = \lim\limits_{k \to
  \infty}y_{\epsilon_k}^\prime (0)=0$.

Finally from the phase portrait analysis, the solution $(y(r),w(r))$
exhibits an oscillatory behavior in $\mathbb{R}^2$; however, note that
in the non-autonomous case, the solution curve does not lie in any
curves of constant Hamiltonian due to the $y^\prime$ term. Since every
closed curve of constant Hamiltonian contains the global minimum
$(\sinh^{-1}\Big(\frac{\lambda}{\tilde{\kappa}^2}\Big),0)$, there
exists $\{R_n\}_{n=1}^\infty$ such that $0<R_n < R_{n+1}, R_n
\xrightarrow[n\to \infty]{}\infty$ such that $y(R_n)=0$.

\end{proof}

\begin{remark}
	For $d=3$, \cref{nonlinearBessel} with $\tilde{\kappa}=1$ can
        be understood as a nonlinear zeroth-order spherical Bessel
        equation. To be more precise, a (linear) zeroth-order
        spherical Bessel ODE is given by
	\[
		r y^{\prime\prime} + 2 y^\prime + r y=0.    
	\]
	The two linearly independent solutions are given by
	\[
		j_0(r) = \frac{\sin r}{r};\: y_0(r) = -\frac{\cos r}{r}.
	\]
	We give plots comparing the linear and nonlinear solutions for $d=3$.
\end{remark}
	
	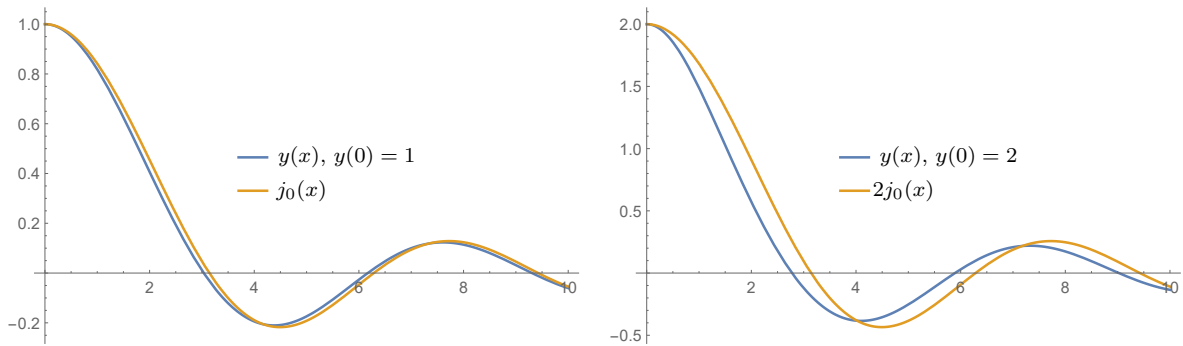
\begin{figure}[ht]
		\centering
\begin{tikzpicture}
        \footnotesize
     \node at (0,0) {\includegraphics[scale=0.30]{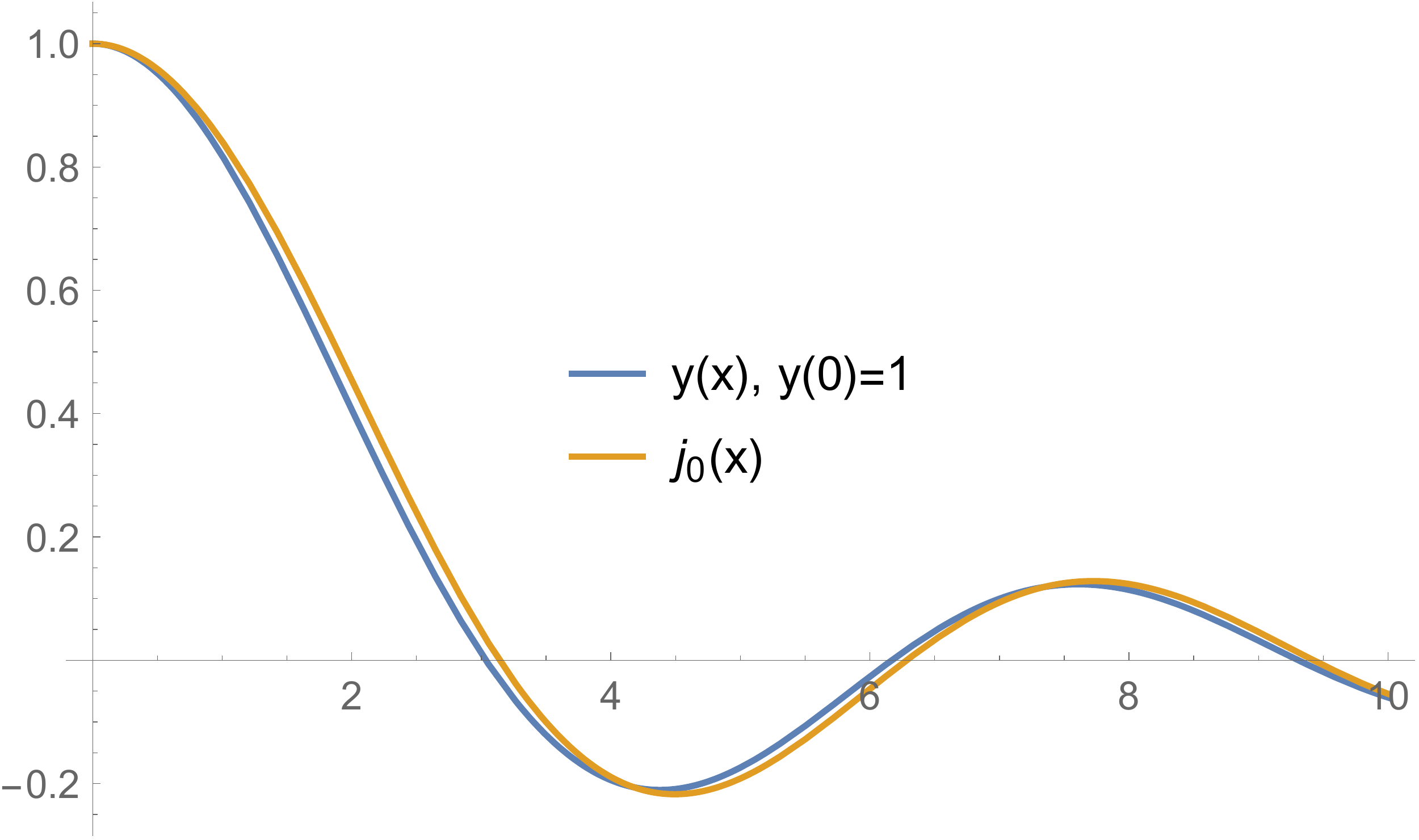}};
     \node at (8,0) {\includegraphics[scale=0.30]{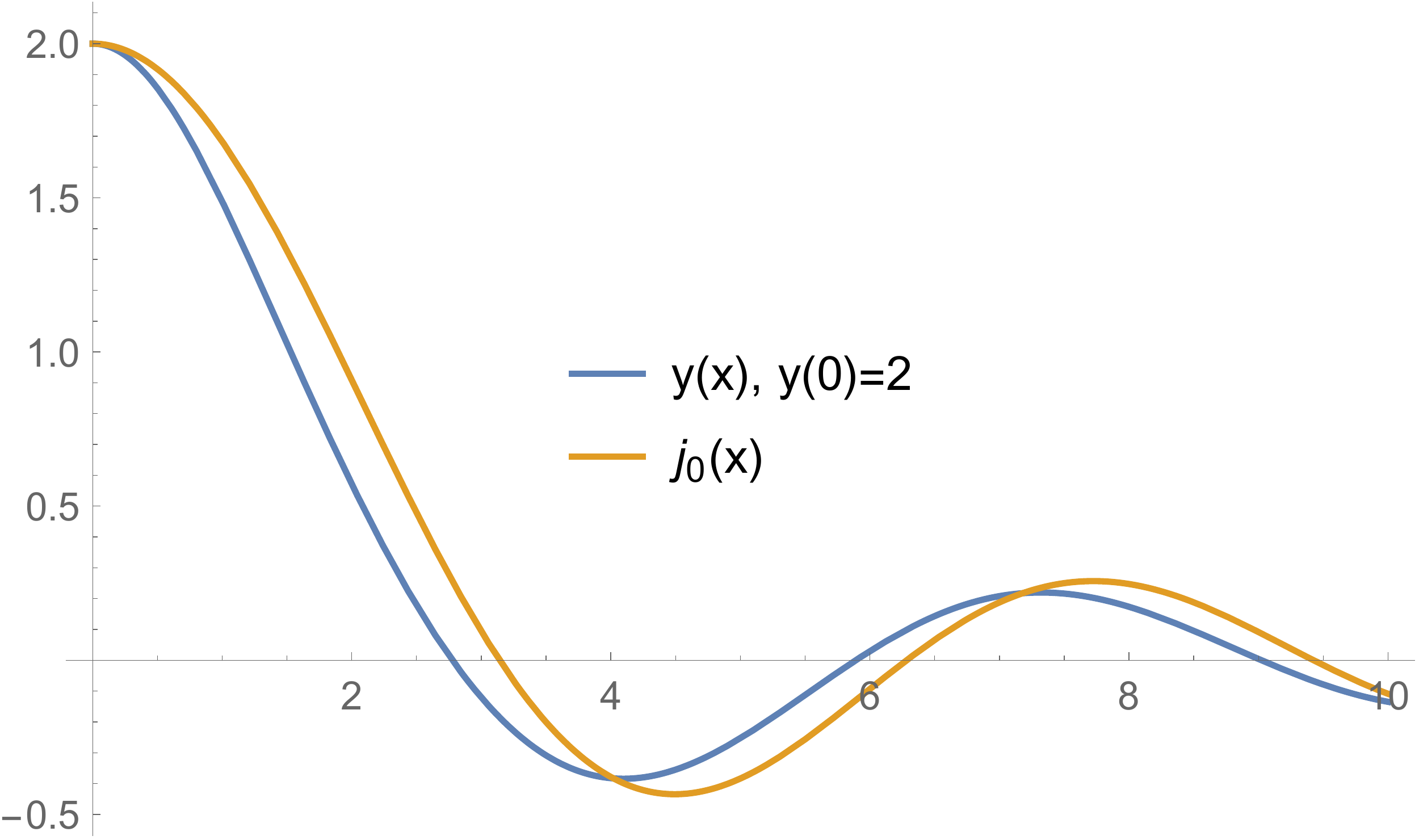}};
     \fill [white] (-0.3,-0.5) rectangle (1.25,0.5);
     \node at (0.7,0.25) {$y(x)$, $y(0)=1$};
     \node at (0.1,-0.2) {$j_0(x)$};
     \fill [white] (7.7,-0.5) rectangle (9.25,0.5);
     \node at (8.7,0.25) {$y(x)$, $y(0)=2$};
     \node at (8.1,-0.2) {$2j_0(x)$};
\end{tikzpicture}
		\caption{Comparison of solutions to \cref{nonlinearBessel} and its corresponding linearization.}
	\end{figure}

\begin{remark}
    Intuitively, this non-uniqueness stems from the non-coercivity of
    the nonlinear operator $Tu \coloneqq -\epsilon\Delta u +
    \kappa^2\sinh u$ when $\epsilon,\kappa$ do not satisfy
    $\epsilon,\kappa>0$. Indeed, our choice of nonlinearity is beyond
    the scope of those discussed in \cite{mcleod1993uniqueness} that
    studies the uniqueness of radial solution to $\Delta u + f(u)=0$
    when $f^\prime(0)<0$. As a simple example, let
    $\epsilon=1,\kappa=i$ and consider the linearized equation
    $u^{\prime\prime} + u =0$ in $x \in (0,\pi)$ with the boundary
    condition $u(0)=u(\pi)=0$. Then, we have an uncountable family of
    solutions $\{A\sin x\}_{A \in \mathbb{R}}$.
    
    However, it turns out that uniqueness can be salvaged if we drop
    the lower orders terms of $\sinh(u)$. We state a result whose
    proof, based on the Derrick-Pohozaev identity, is easily adapted
    from that of \cite[Section 9.4, Theorem 1]{evans2010partial};
    compare this to \cref{nonunique}.
    
    Let $d\geq 3$ and $N_0>\frac{d+2}{d-2}$ be an odd integer. Suppose
    $u \in C^2(\overline{\Omega})$ is a classical solution to
    	\[
    	\begin{split}
    	\label{counterexample2}
		-\Delta u &= \sum_{N \geq N_0,\: N\:odd} \frac{u^N}{N!},\: x \in \Omega\\
		u &=0,\: x \in \partial \Omega,\nonumber
	\end{split}
	\]
    where $\Omega$ is a star-shaped domain containing $0 \in
    \mathbb{R}^d$ with $\partial \Omega \in C^1$. Then, $u=0$ in
    $\overline{\Omega}$.
\end{remark}
\end{appendices}

\bibliographystyle{abbrv}
\bibliography{pdesref,citations}

\begin{thebibliography}{10}

\bibitem{babusk_nobile_temp_10}
I.~Babuska, F.~Nobile, and R.~Tempone.
\newblock A stochastic collocation method for {Elliptic Partial Differential
  Equations} with random input data.
\newblock {\em SIAM Review}, 52(2):317--355, 2010.

\bibitem{Bajaj2005}
C.~Bajaj and Z.~Yu.
\newblock Geometric and signal processing of reconstructed 3d maps of molecular
  complexes.
\newblock In S.~Aluru, editor, {\em Handbook of Computational Molecular
  Biology}. Chapman \& Hall/CRC Press, 2005.

\bibitem{Bajaj2003}
C.~Bajaj, Z.~Yu, and M.~Auer.
\newblock Volumetric feature extraction and visualization of tomographic
  molecular imaging.
\newblock {\em Journal of Structural Biology}, 144(1):132 -- 143, 2003.
\newblock Analytical Methods and Software Tools for Macromolecular Microscopy.

\bibitem{Baker2001}
N.~Baker, D.~Sept, S.~Joseph, M.~Holst, and J.~McCammon.
\newblock Electrostatics of nanosystems: application to microtubules and the
  ribosome.
\newblock {\em Proceedings of the National Academy of Sciences of the United
  States of America}, 98(18):10037—10041, August 2001.

\bibitem{Berman2000}
H.~M. Berman, J.~Westbrook, Z.~Feng, G.~Gilliland, T.~Bhat, H.~Weissig,
  I.~Shindyalov, and P.~Bourne.
\newblock The protein data bank (www.pdb.org).
\newblock {\em Nucleic Acids Res.}, 28:235--242, 2000.

\bibitem{brezis2010functional}
H.~Brezis.
\newblock {\em Functional analysis, Sobolev spaces and partial differential
  equations}.
\newblock Springer Science \& Business Media, 2010.

\bibitem{brezis2007nonlinear}
H.~Brezis, M.~Marcus, and A.~C. Ponce.
\newblock Nonlinear elliptic equations with measures revisited.
\newblock {\em Mathematical Aspects of Nonlinear Dispersive Equations (J.
  Bourgain, C. Kenig, and S. Klainerman, eds.), Annals of Mathematics Studies},
  163:55--110, 2007.

\bibitem{Castrillon2020}
J.~E. Castrill{\'o}n-Cand{\'a}s and M.~Kon.
\newblock Analytic regularity and stochastic collocation of high-dimensional
  newton iterates.
\newblock {\em Advances in Computational Mathematics}, 46(3):42, May 2020.

\bibitem{Castrillon2016}
J.~E. Castrill\'{o}n-Cand\'{a}s, F.~Nobile, and R.~Tempone.
\newblock Analytic regularity and collocation approximation for {PDEs} with
  random domain deformations.
\newblock {\em Computers and Mathematics with applications}, 71(6):1173--1197,
  2016.

\bibitem{Castrillon2021}
J.~E. Castrillón-Candás and J.~Xu.
\newblock A stochastic collocation approach for parabolic pdes with random
  domain deformations.
\newblock {\em Computers \& Mathematics with Applications}, 93:32--49, 2021.

\bibitem{crandall1971bifurcation}
M.~G. Crandall and P.~H. Rabinowitz.
\newblock Bifurcation from simple eigenvalues.
\newblock {\em Journal of Functional Analysis}, 8(2):321--340, 1971.

\bibitem{evans2010partial}
L.~C. Evans.
\newblock {\em Partial differential equations}, volume~19.
\newblock American Mathematical Soc., 2010.

\bibitem{gilbarg2015elliptic}
D.~Gilbarg and N.~S. Trudinger.
\newblock {\em Elliptic partial differential equations of second order}.
\newblock springer, 2015.

\bibitem{Holst1994}
M.~Holst.
\newblock {\em The {Poisson-Boltzmann} equation: Analysis and multilevel
  numerical solution.}
\newblock Applied Mathematics and CRPC, California Institute of Technology, 1st
  ed edition, 1994.

\bibitem{kielhofer2011bifurcation}
H.~Kielh{\"o}fer.
\newblock {\em Bifurcation theory: an introduction with applications to partial
  differential equations}, volume 156.
\newblock Springer Science \& Business Media, 2011.

\bibitem{kwong1989uniqueness}
M.~K. Kwong.
\newblock Uniqueness of positive solutions of $\delta$u- u+ up= 0 in r n.
\newblock {\em Archive for Rational Mechanics and Analysis}, 105(3):243--266,
  1989.

\bibitem{mclean2000strongly}
W.~McLean and W.~C.~H. McLean.
\newblock {\em Strongly elliptic systems and boundary integral equations}.
\newblock Cambridge university press, 2000.

\bibitem{mcleod1993uniqueness}
K.~McLeod.
\newblock Uniqueness of positive radial solutions of ${\Delta} u + f(u)=0$ in
  $\mathbb{R}^n$ {II}.
\newblock {\em Transactions of the American Mathematical Society},
  339(2):495--505, 1993.

\bibitem{mizuguchi2017estimation}
M.~Mizuguchi, K.~Tanaka, K.~Sekine, and S.~Oishi.
\newblock Estimation of sobolev embedding constant on a domain dividable into
  bounded convex domains.
\newblock {\em Journal of inequalities and applications}, 2017(1):299, 2017.

\bibitem{Neumaier1997}
A.~Neumaier.
\newblock Molecular modeling of proteins and mathematical prediction of protein
  structure.
\newblock {\em SIAM Rev.}, 39(3):407–460, Sept. 1997.

\bibitem{nobile_tempone_08}
F.~Nobile and R.~Tempone.
\newblock Analysis and implementation issues for the numerical approximation of
  parabolic equations with random coefficients.
\newblock {\em International Journal for Numerical Methods in Engineering},
  80(6-7):979--1006, 2009.

\bibitem{nobile2008a}
F.~Nobile, R.~Tempone, and C.~Webster.
\newblock A sparse grid stochastic collocation method for partial differential
  equations with random input data.
\newblock {\em SIAM Journal on Numerical Analysis}, 46(5):2309--2345, 2008.

\bibitem{Padhorny2016}
D.~Padhorny, A.~Kazennov, B.~S. Zerbe, K.~A. Porter, B.~Xia, S.~E. Mottarella,
  Y.~Kholodov, D.~W. Ritchie, and D.~Kozakov.
\newblock Protein-protein docking by fast generalized fourier transforms on 5d
  rotational manifolds.
\newblock {\em Proceedings of the National Academy of Sciences of the United
  States of America}, 113:4286--4293, 2016.

\bibitem{payne1960optimal}
L.~E. Payne and H.~F. Weinberger.
\newblock An optimal poincar{\'e} inequality for convex domains.
\newblock {\em Archive for Rational Mechanics and Analysis}, 5(1):286--292,
  1960.

\end{thebibliography}
\end{document}